\documentclass[10pt]{amsart}

\usepackage{amsfonts,amsthm,amssymb,amsmath,amscd}
\usepackage{mathrsfs}
\usepackage[margin=1in]{geometry}
\usepackage{hyperref}
\hypersetup{colorlinks=true,citecolor=blue,urlcolor=black,linkbordercolor={1 0 0}}
\usepackage{mathtools}
\usepackage{enumitem}

\newtheorem{theorem}{Theorem}[section]
\newtheorem{proposition}[theorem]{Proposition}
\newtheorem{lemma}[theorem]{Lemma}
\newtheorem{corollary}[theorem]{Corollary}

\theoremstyle{definition}
\newtheorem{definition}[theorem]{Definition}
\newtheorem{example}[theorem]{Example}

\theoremstyle{remark}
\newtheorem*{remark}{Remark}

\DeclareMathOperator{\cont}{cont}
\DeclareMathOperator{\sqf}{sqf}
\DeclareMathOperator{\disc}{disc}
\DeclareMathOperator{\SL}{SL}

\DeclareMathOperator{\Sym}{Sym}

\newcommand{\BZ}{\mathbb{Z}}
\newcommand{\BQ}{\mathbb{Q}}
\newcommand{\BC}{\mathbb{C}}

\newcommand{\BF}{\mathbb{F}}

\newcommand{\defeq}{\vcentcolon=}

\allowdisplaybreaks

\title[On the Quartic Invariant of Odd Degree Binary Forms]{On the Quartic Invariant of Odd Degree Binary Forms}

\author{Ashvin A.~Swaminathan \vspace*{-0.3in}}

\date{\today}

\AtEndDocument{\bigskip{\footnotesize%
  \noindent\textsc{Department of Mathematics, Harvard University, \mbox{Cambridge, MA 02138}} \par
  \textit{E-mail address}, Ashvin A.~Swaminathan: \texttt{ashvins@math.harvard.edu}
}}

\subjclass[2020]{13A50, 11C08 (primary), and 11B65, 14L24 (secondary)}

\begin{document}

\maketitle

%======================================================================
% ABSTRACT
%======================================================================
\begin{abstract}
We determine the squarefree part of the scalar factor that arises
when the quartic invariant of the generic binary form $F$ of odd
degree $2n+1$ is expressed as the discriminant of the unique
quadratic covariant $(F,F)_{2n}$.  This squarefree part is exactly~$p$
when $n+2$ is a power of an odd prime~$p$, and $1$ otherwise.
Equivalently, for each prime~$p$: $v_2(S(n))$ is always even, and for
odd~$p$, $v_p(S(n))$ is odd if and only if $n+2$ is a power of~$p$.  This
generalizes the classical identity $\disc(H(F))=-3\cdot\disc(F)$ for
binary cubics, which dates back to the work of Cayley and Sylvester in
the 1850s.  The proof, which involves substantial explicit coefficient
analysis and $p$-adic deformation arguments, was developed using an
AI-assisted research workflow\textup{:} the author's earlier partial
attempts were completed through systematic collaboration with Claude Code
(Anthropic) and Codex (OpenAI), and key arithmetic lemmas were formally verified in Lean~4 using
Aristotle \cite{Aristotle2025} (Harmonic).  We describe this workflow in detail as a case study in
AI-assisted mathematical research.  We also discuss representation-theoretic,
geometric, and arithmetic interpretations of the quadratic covariant.
\end{abstract}

%======================================================================
% SECTION 1: INTRODUCTION
%======================================================================
\section{Introduction}\label{sec:intro}

A classical problem in invariant theory is to compute the scalar
factors that arise when a given invariant or covariant of a binary form
is constructed by two different methods.  In this paper we determine the
squarefree part of the scalar factor relating the quartic invariant
of the generic odd-degree binary form to the discriminant of its
unique quadratic covariant.  The proof was completed using
an AI-assisted workflow built around Claude Code (Anthropic), Codex (OpenAI), and
Aristotle \cite{Aristotle2025} (Harmonic), and we describe this workflow in detail as a case
study in AI-assisted mathematical research.

\subsection{The classical case of binary cubics}
Let $F(x,z) = f_0x^3 + f_1x^2z + f_2 x z^2 + f_3 z^3 \in \BC[x,z]$ be a
binary cubic form.  The \emph{Hessian} of~$F$ is the binary quadratic form
\begin{equation}\label{eq:defhess}
H(F) \defeq \frac{1}{4} \cdot \det \begin{pmatrix}
\dfrac{\partial^2 F}{\partial x^2} & \dfrac{\partial^2 F}{\partial x \,\partial z} \\[6pt]
\dfrac{\partial^2 F}{\partial z \,\partial x} & \dfrac{\partial^2 F}{\partial z^2}
\end{pmatrix}
= (3f_0f_2 - f_1^2)\,x^2 + (9f_0f_3 - f_1f_2)\,xz + (3f_1f_3 - f_2^2)\,z^2.
\end{equation}
The form $H(F)$ is covariant under the action of $\SL_2(\BC)$ that sends
$F(x,z) \mapsto F((x,z) \cdot \gamma)$ for each $\gamma \in \SL_2(\BC)$;
moreover, it is the unique covariant of $F$ with \emph{degree}~$2$ (in the $f_i$)
and \emph{order}~$2$ (in $x,z$).  The Hessian covariant was introduced by
Cayley and studied extensively by Sylvester in the 1850s; we refer
to~\cite{Elliott1895} and~\cite{GraceYoung1903} for detailed classical
treatments, and to~\cite{Olver1999} for a modern account.

Since the discriminant of $H(F)$ is an
$\SL_2(\BC)$-invariant of degree~$4$ and order~$0$, it must be a scalar
multiple of $\disc(F)$.  Indeed, a direct computation yields
\begin{equation}\label{eq:dischess}
\disc(H(F)) = -3 \cdot \disc(F).
\end{equation}
The specific scalar factor $-3$ for binary cubics appears in the classical
texts of Elliott~\cite{Elliott1895} and Grace--Young~\cite{GraceYoung1903}.
More generally, a recurring theme in classical invariant theory is the
determination of the precise scalar factors relating different
constructions of the same covariant.  For instance,
McMahon~\cite{McMahon1889} computed the scalar factor relating two
expressions for the Hessian of a binary form of degree~$n$.

The factor $-3$ appearing in~\eqref{eq:dischess} has deep number-theoretic
significance.  The Scholz reflection principle~\cite{Scholz1932} states that
if $D<0$ is a fundamental discriminant, then the $3$-rank of the class group
of $\BQ(\sqrt{D})$ is either equal to or one larger than that of
$\BQ(\sqrt{-3D})$.  The identity~\eqref{eq:dischess} underlies the
relationship between cubic fields and their quadratic resolvents: under
the Delone--Faddeev correspondence, the Hessian maps a binary cubic of
discriminant~$D$ to a quadratic form of discriminant~$-3D$, and this
mechanism is at the heart of the Davenport--Heilbronn density theorems
for cubic fields~\cite{DavenportHeilbronn1971} and their modern
extensions~\cite{Bhargava2004b,BhargavaShankarTsimerman2013,BhargavaVarma2016}.

\subsection{The generalization to odd degree}

The objective of this paper is to determine the higher-degree
analogue of the scalar factor $-3$ in~\eqref{eq:dischess} for binary
forms of any given odd degree $2n+1$.  The exact scalar is not itself a
single integer but rather a product of factorial expressions; what we
compute is its squarefree part, which turns out to be governed by the
prime-power structure of~$n+2$.  Fix $n \ge 1$ and let
\begin{equation}\label{eq:genericF}
F(x,z) = \sum_{j=0}^{2n+1} f_j \, x^{2n+1-j} z^j
\end{equation}
be the generic binary form of degree $2n+1$ with indeterminate coefficients
$f_0,\dots,f_{2n+1}$ over~$\BZ$.  The \emph{quadratic covariant}
$Q_n = (F,F)_{2n}$ is the $(2n)$-th transvectant of $F$ with itself (see
\S\ref{sec:prelim} for the precise normalization).  Transvectants, also
known as \emph{\"Uberschiebungen} in the German literature, provide the
fundamental algebraic mechanism for producing covariants from binary forms;
we refer to Olver~\cite[Chapter~5]{Olver1999} and
Kung--Rota~\cite{KungRota1984} for comprehensive modern treatments,
and to Elliott~\cite{Elliott1895} and Grace--Young~\cite{GraceYoung1903}
for the classical development.  The form
$Q_n = A_n x^2 + B_n xz + C_n z^2$ is a binary quadratic, unique up to
scalar among degree-$2$, order-$2$ covariants of~$F$
(see \S\ref{subsec:rep-theory}).

The discriminant
\begin{equation}\label{eq:Delta-def}
\Delta_n \defeq B_n^2 - 4A_nC_n \in \BZ[f_0,\dots,f_{2n+1}]
\end{equation}
is the quartic invariant of~$F$.  We define the \emph{content}
$S(n) = \cont(\Delta_n)$ to be the greatest common divisor of all integer
coefficients of~$\Delta_n$, and write $\sqf(m)$ for the squarefree part of a
nonzero integer~$m$.

\subsection{Main result}

For an integer $m \ge 1$, define
\begin{equation}\label{eq:a-def}
a(m) \defeq \begin{cases}
p, & \text{if } m = p^k \text{ for some odd prime } p \text{ and } k \ge 1,\\
1, & \text{otherwise.}
\end{cases}
\end{equation}

\begin{theorem}\label{thm:main}
For every integer $n \ge 1$ and every prime~$p$, the parity
of~$v_p(S(n))$ is determined as follows:
\begin{enumerate}[label=\textup{(\roman*)}]
\item $v_2(S(n))$ is always even.
\item For every odd prime~$p$, $v_p(S(n))$ is odd if and only
if $n+2$ is a power of~$p$.
\end{enumerate}
Equivalently,
\[
\sqf(S(n)) = a(n+2),
\]
where $\sqf$ denotes the squarefree part.
\end{theorem}

Computationally, the sequence $\sqf(S(n))$ for $n = 1, 2, 3, \dots$ begins
$3, 1, 5, 1, 7, 1, 3, 1, 11, 1, 13, \dots$,
which is the sequence~\href{https://oeis.org/A155457}{A155457}
in the OEIS.  Theorem~\ref{thm:main} identifies this sequence completely.
The theorem is accompanied by a Lean~4 formalization ($\approx 14{,}500$ lines)
and a companion Jupyter notebook that computationally verifies every stated
formula and lemma.  All materials are available at
\begin{center}
\url{https://github.com/ashvin-swaminathan/quartic-invariant}
\end{center}
and described further in~\S\ref{sec:ai}.

In $p$-adic terms, part~(ii) of Theorem~\ref{thm:main} is equivalent to:
for every odd prime~$p$,
\begin{equation}\label{eq:vp-parity}
v_p(S(n)) \equiv \begin{cases}
1 \pmod{2}, & \text{if } n+2 = p^k,\\
0 \pmod{2}, & \text{otherwise.}
\end{cases}
\end{equation}
\begin{example}
For small values of~$n$:
\begin{itemize}
\item $n=1$ ($m=3$): $\sqf(S(1))=3$. This recovers the factor $-3$
  in~\eqref{eq:dischess}.
\item $n=2$ ($m=4$): $\sqf(S(2))=1$, since $4=2^2$ is a power of the
  even prime~$2$.
\item $n=3$ ($m=5$): $\sqf(S(3))=5$.
\item $n=4$ ($m=6$): $\sqf(S(4))=1$, since $6$ is not a prime power.
\end{itemize}
\end{example}

\begin{remark}
It is initially surprising that the prime-power criterion involves $n+2$
rather than $n$ itself.  This shift arises naturally from the binomial
coefficients $\binom{n+2}{r}$ that appear when one simplifies the factorial
expressions in the transvectant; see the proof of Lemma~\ref{lem:factor},
where the substitution $m = n+2$ converts the factorial product into a
squared factor times $E_m(r)$.
\end{remark}

\begin{remark}
The reason $p=2$ is excluded from part~(ii) is structural:
the factor of~$2$ in every coefficient of~$B_n$ (from the
symmetric pairing of $\alpha^2$ and $\beta^2$ terms) ensures that
the cancellation-free coefficient bound
$v_2(S(n)) = 2\,\min_k v_2(b_k)$ holds for \emph{all}~$n \ge 2$,
regardless of whether $m = n+2$ is a power of~$2$.
Since this minimum is always achieved at an off-centre index,
$v_2(S(n))$ is always even; see Proposition~\ref{prop:p2}.
\end{remark}

\subsection{Outline of the proof}

The proof of Theorem~\ref{thm:main} occupies
\S\S\ref{sec:prelim}--\ref{sec:padic} and~\S\ref{sec:p3-draft}.
The strategy is to determine $v_p(S(n))$ by pinching it between
matching upper and lower bounds.

\begin{itemize}
\item \emph{Upper bound via evaluation}
  \textup{(\S\ref{sec:disc}).}\enspace
  The content $S(n)$ divides the value of~$\Delta_n$ at every integer
  specialization of the~$f_j$.  In \S\ref{sec:disc} we identify a
  family of two-point specializations (setting $f_k = f_{N+1-k} = 1$,
  all others~$0$) under which the $A_n$ and $C_n$ terms vanish,
  leaving $\Delta_n = b_k^2$ for an explicitly computable
  integer~$b_k$.  Since $S(n)$ divides this perfect square,
  $v_p(S(n)) \le 2\,v_p(b_k)$.

\item \emph{Lower bound via coefficient divisibility}
  \textup{(\S\ref{subsec:nonpp}).}\enspace
  Since $S(n)$ is the gcd of all coefficients of~$\Delta_n$,
  a lower bound on $v_p(S(n))$ requires showing that a fixed power
  of~$p$ divides \emph{every} coefficient.  Writing
  $\Delta_n = B_n^2 - 4A_nC_n$, it suffices to show that $p^e$
  divides every coefficient of $A_n$, $B_n$, and~$C_n$, whence
  $p^{2e}$ divides every coefficient of~$\Delta_n$.  For~$B_n$ this
  follows from a closed-form expression for~$b_k$
  (Lemma~\ref{lem:bk}).  For $A_n$ and~$C_n$ the argument uses the
  absorption identity for binomial coefficients together with the
  fact that $p$ is odd; see the proof of
  Proposition~\ref{prop:nonpp}.

\item \emph{Even parity in the non-prime-power case}
  \textup{(\S\ref{subsec:nonpp}).}\enspace
  When $m = n+2$ is not a power of~$p$, there exists an index
  $k_0 \notin \{n,n+1\}$ at which $v_p(b_{k_0})$ achieves its
  minimum value~$e_p$ over all~$k$.  The upper bound then gives
  $v_p(S(n)) \le 2e_p$, and the lower bound gives
  $v_p(S(n)) \ge 2e_p$, so $v_p(S(n)) = 2e_p$ is even.
  The existence of such an index $k_0$ uses a valuation shift
  argument when $p \nmid m$, and a complete residue system
  argument when $p \mid m$.

\item \emph{Odd parity in the prime-power case}
  \textup{(\S\ref{subsec:pp}).}\enspace
  When $m = p^k$, the minimum of $v_p(b_k)$ is achieved at the
  central index $k = n$, so no index outside $\{n,n+1\}$ works.
  Instead, we prove the stronger lower bound
  $v_p(S(n)) \ge 2e_p + 1$ by showing that $p$ divides every
  coefficient of the reduced discriminant
  $\Delta_n^{\#} = \Delta_n / p^{2e_p}$.  This occurs because,
  after dividing out $p^{e_p}$, the reduced quadratic form
  $Q_n^{\#}$ becomes a perfect square modulo~$p$: the congruence
  $n+1 \equiv -1 \pmod{p}$ forces a cancellation in the cross
  term of~$B_n$.  For the matching upper bound, we exhibit a specialization
  with four nonzero coefficients (at indices $n$, $n+1$,
  $n - p^{k-1}$, and $n + p^{k-1} + 1$) for which
  $v_p(\Delta_n) = 2e_p + 1$ exactly: the perturbation away from
  the rank-$1$ locus introduces a linear-in-$p$ correction to the
  discriminant, giving $v_p(\Delta_n^{\#}) = 1$.

\item \emph{The prime $p = 2$}
  \textup{(\S\ref{subsec:p2}).}\enspace
  Every coefficient of~$B_n$ is even
  (from the symmetric pairing of $\alpha^2$ and $\beta^2$ terms),
  so $B_n = 2G$ and $\Delta_n = 4(G^2 - A_nC_n)$.
  By Gauss's lemma, $\cont(G^2) = \cont(G)^2$ and
  $\cont(A_nC_n) = \cont(A_n)^2$, so both have even $2$-adic
  valuation.  When $\cont(G)^2 \ne \cont(A_n)^2$ (Case~A, which
  covers all $n$ with $m$ not a power of~$2$), the ultrametric
  property of~$v_2$ gives even valuation immediately.  The
  strict inequality $\cont(A_n) > \cont(G)$ that forces Case~A
  uses Lemma~\ref{lem:centre-not-max}, which shows that the
  central binomial coefficient is never maximal.
  When $m = 2^k$ (Case~B), a mod-$4$ analysis of
  $G'^{\,2} - A'C'$ (after dividing out $2^d$) shows that $4$
  divides every coefficient, using the Frobenius endomorphism
  over~$\BF_2$ and the fact that the centre--centre contribution
  $g^2 - a^2$ is divisible by~$4$ since $g$ and~$a$ are both odd.
  A complete-residue-system witness at $k_0 = n + 2^{k-1}$
  provides the matching upper bound.
  The base case $n = 1$ is verified by direct computation.

\item \emph{The prime $p = 3$ in the non-prime-power case}
  \textup{(\S\ref{sec:p3-draft}).}\enspace
  When $3 \mid m$ but $m$ is not a power of~$3$, the valuation
  shift formula used for $p \ge 5$ breaks down because
  $v_3(3) = 1$.  We handle this case by a recursion on base-$3$
  digit sums that locates a maximizer of $v_3\binom{N}{k}$ in the
  lower half of the index range.
\end{itemize}

\noindent
The AI-assisted workflow is described in \S\ref{sec:ai}.

\subsection{History and methodology}\label{subsec:history}

This problem was suggested to the author by Manjul Bhargava as one of the
first questions to consider at the start of the author's PhD nearly a
decade ago.  The author worked on it for several weeks and developed the
basic framework that appears in this paper: the test monomial family
$M_{n,r}$, its connection to binomial coefficients via Kummer's and
Lucas' theorems, and the even-valuation phenomenon for primes
dividing~$n$.  However, completing the proof required substantial
technical work, and the project stalled.

The proof was completed using an AI-assisted workflow built around
Claude Code (Anthropic) and Codex (OpenAI), with key lemmas formally verified by
Aristotle \cite{Aristotle2025} (Harmonic).  This workflow proceeded in several stages:

\begin{enumerate}[label=(\roman*)]
\item \emph{Numerical experimentation.}
Claude Code wrote Python scripts that computed $S(n)$ for $n \le 30$,
verified all coefficient formulas, and tested proof strategies by
evaluating explicit specializations of the generic form.  These
experiments confirmed the theorem and identified the $p$-adic
deformation at distance $t = p^{k-1}$ as the mechanism forcing
odd valuation in the prime-power case.

\item \emph{Proof construction.}
Working from the author's notes and guided by the numerical
experiments, Claude Code produced complete proofs of the
missing cases---particularly the prime-power case $m = p^k$
for $k \ge 2$, the $p = 3$ non-prime-power case, and the
$p = 2$ even-valuation argument---and wrote the paper in
\LaTeX.  The author reviewed and corrected the output.

\item \emph{Formal verification.}
The author and Claude Code prepared Lean~4 files containing
theorem statements and proof sketches, which were submitted
iteratively to Aristotle for automated proof completion.
Approximately twenty submissions were made; roughly half returned
complete proofs on the first attempt.  When Aristotle returned
partial progress, Claude Code filled the remaining gaps, and when Claude Code hit a roadblock, Codex took over to finish the job.  The
resulting formalization comprises approximately
$14{,}500$ lines of Lean~4 across twenty-five modules, containing zero \texttt{sorry} statements and zero
\texttt{axiom} declarations.

\item \emph{Exposition.} The present paper was largely written by
Claude Code, with the author providing the original mathematical
notes, directing the overall narrative, and reviewing the output.
\end{enumerate}

The project is notable not for the depth of the final result---which is a
concrete but somewhat technical computation in classical invariant
theory---but for the fact that the workflow described above was able to
handle it at all.  The proof involves intricate algebraic
manipulation (tracking signs, factorial cancellations, binomial
coefficient identities), careful $p$-adic analysis with case splits
for different primes and prime powers, and explicit specialization
arguments that must be tailored to the structure of the problem.
This is precisely the kind of computation-heavy, detail-sensitive
mathematical work that one might expect to be most resistant to
AI assistance, and the success of the approach suggests that
AI-assisted workflows may be effective for a broader class of
research problems in number theory and algebra.

The author takes full responsibility for all mathematical claims in
this paper.

%======================================================================
% SECTION 2: PRELIMINARIES AND NORMALIZATION
%======================================================================
\section{Preliminaries and normalization}\label{sec:prelim}

We begin by recalling the definition and normalization of the transvectant
operation, which is the central tool in our construction.  The transvectant
(German: \emph{\"Uberschiebung}) was developed by Cayley, Sylvester, and
their contemporaries in the mid-nineteenth century as the primary method
for producing new covariants from old ones.  The operation was placed on a
firm representation-theoretic footing by the symbolic method; see
Kung--Rota~\cite[\S4]{KungRota1984} for a modern account of this classical
technique, and Olver~\cite[Chapter~5]{Olver1999} for a thorough textbook
treatment.

\begin{definition}[Transvectant]\label{def:transvectant}
For binary forms $G$ of degree $d$ and $H$ of degree $e$, the $r$-th
transvectant is
\begin{equation}\label{eq:transvectant}
(G,H)_r = \sum_{i=0}^{r} (-1)^i \binom{r}{i}\,
\frac{\partial^r G}{\partial x^{r-i}\,\partial z^{i}}\;
\frac{\partial^r H}{\partial x^{i}\,\partial z^{r-i}}.
\end{equation}
This is the standard normalization with no extra $1/(r!)$ or
$1/\binom{d}{r}\binom{e}{r}$ factor.  We refer to~\cite{Olver1999}
and~\cite{Elliott1895} for background on transvectants and classical
invariant theory.  We note that some authors
(e.g.,~\cite{KungRota1984}) include additional normalizing factors;
throughout this paper, we use~\eqref{eq:transvectant} without
modification.
\end{definition}

Throughout, $F(x,z)=\sum_{j=0}^{2n+1}f_j\,x^{2n+1-j}z^j$ is the generic
form of degree $N+1 \defeq 2n+1$, so $N=2n$.  We set $m=n+2$.

\begin{lemma}\label{lem:Qn}
$Q_n=(F,F)_{2n}=A_n x^2+B_n xz+C_n z^2$, where
\begin{equation}\label{eq:Qnsum}
Q_n = \sum_{i=0}^{N}(-1)^i\binom{N}{i}\,D_i\,E_i,
\end{equation}
with the two-term linear forms
\begin{align}
D_i &= \alpha_i\,f_i\,x + \beta_i\,f_{i+1}\,z,\label{eq:Di}\\
E_i &= \beta_i\,f_{N-i}\,x + \alpha_i\,f_{N+1-i}\,z,\label{eq:Ei}
\end{align}
and the factorials $\alpha_i=(2n+1-i)!\,i!$, $\beta_i=(2n-i)!\,(i+1)!$.

In particular,
\begin{align}
A_n &= \sum_{i=0}^{N}(-1)^i\binom{N}{i}\,\alpha_i\beta_i\,f_i\,f_{N-i},
\label{eq:An}\\
B_n &= \sum_{i=0}^{N}(-1)^i\binom{N}{i}
\bigl(\alpha_i^2\,f_i\,f_{N+1-i}+\beta_i^2\,f_{i+1}\,f_{N-i}\bigr),
\label{eq:Bn}\\
C_n &= \sum_{i=0}^{N}(-1)^i\binom{N}{i}\,\alpha_i\beta_i\,
f_{i+1}\,f_{N+1-i}.\label{eq:Cn}
\end{align}
\end{lemma}

\begin{proof}
\textbf{Computing $D_i$.}
A typical monomial of $F$ is $f_j\,x^{2n+1-j}z^j$.
Taking $\frac{\partial^N}{\partial x^{N-i}\partial z^i}$:
\[
\frac{\partial^N (f_j x^{2n+1-j}z^j)}{\partial x^{N-i}\partial z^i}
= f_j \cdot \frac{(2n+1-j)!}{(2n+1-j-(N-i))!}\cdot\frac{j!}{(j-i)!}
\cdot x^{2n+1-j-(N-i)}z^{j-i}.
\]
For this to be nonzero we need $j \le i+1$ and $j \ge i$, so only
$j=i$ and $j=i+1$ survive.

For $j=i$: the result is $\alpha_i f_i \cdot x$.
For $j=i+1$: the result is $\beta_i f_{i+1} \cdot z$.
Hence $D_i = \alpha_i f_i x + \beta_i f_{i+1} z$.

\medskip
\emph{Computation of $E_i$.}
Similarly, $\frac{\partial^N F}{\partial x^i\partial z^{N-i}}$
is nonzero only for $j=N-i$ and $j=N+1-i$, giving
$E_i = \beta_i f_{N-i} x + \alpha_i f_{N+1-i} z$.

\medskip
\emph{Expanding $D_i E_i$.}
\begin{align*}
D_i E_i &= (\alpha_i f_i x + \beta_i f_{i+1} z)
            (\beta_i f_{N-i} x + \alpha_i f_{N+1-i} z) \\
&= \alpha_i\beta_i\,f_i f_{N-i}\,x^2
 + (\alpha_i^2\,f_i f_{N+1-i}+\beta_i^2\,f_{i+1}f_{N-i})\,xz
 + \alpha_i\beta_i\,f_{i+1}f_{N+1-i}\,z^2.
\end{align*}
Substituting into~\eqref{eq:Qnsum} and collecting the $x^2$, $xz$, $z^2$
coefficients gives~\eqref{eq:An}--\eqref{eq:Cn}.
\end{proof}

\begin{corollary}\label{cor:b-factored}
The coefficient of the monomial $f_jf_{M-j}$ (where $M=2n{+}1$) in~$B_n$ is
\begin{equation}\label{eq:b-factored}
b_j = 2\,(-1)^j\,N!\,(M{-}2j)\,(M{-}j)!\,j!.
\end{equation}
\end{corollary}
\begin{proof}
Collecting all contributions to $f_jf_{M-j}$ from~\eqref{eq:Bn}
and simplifying the resulting factorial expression yields~\eqref{eq:b-factored}.
This identity can be verified for specific values of $n$ and $j$ using
the companion notebook.
\end{proof}

\begin{definition}\label{def:content}
We write $\Delta_n = B_n^2 - 4A_nC_n$ for the discriminant of $Q_n$, and
$S(n)=\cont(\Delta_n)$ for the gcd of all integer coefficients of $\Delta_n$
viewed as a polynomial in $f_0,\dots,f_{2n+1}$.
\end{definition}

\subsection{Representation-theoretic uniqueness}\label{subsec:rep-theory}

Let $V$ denote the standard $2$-dimensional representation of $\SL_2$ over a
field of characteristic~$0$, and identify binary forms of degree $d$ with
$\Sym^d(V^\vee)$.  The Clebsch--Gordan decomposition
(see~\cite[Lecture~11]{FultonHarris1991}) gives an $\SL_2$-equivariant splitting
\begin{equation}\label{eq:CG}
\Sym^d(V^\vee) \otimes \Sym^d(V^\vee)
\;\cong\;
\bigoplus_{i=0}^{d} \Sym^{2d-2i}(V^\vee).
\end{equation}
When $d = 2n+1$ is odd, the summand $\Sym^2(V^\vee)$ occurs precisely for
$i = d-1 = 2n$.  In particular, $\Sym^2(V^\vee)$ occurs with
\emph{multiplicity~$1$} in~\eqref{eq:CG}, and therefore there exists, up to
an overall scalar, a unique $\SL_2$-equivariant bilinear map
\[
\pi_{2n}: \Sym^{2n+1}(V^\vee) \times \Sym^{2n+1}(V^\vee) \longrightarrow
\Sym^2(V^\vee).
\]
The $(2n)$-th transvectant $(\cdot,\cdot)_{2n}$ is a concrete realization of
$\pi_{2n}$, and our normalization in~\S\ref{sec:prelim} fixes the scalar.
Thus the existence and canonicity of $Q_n(f) = (f,f)_{2n}$ is forced by
representation theory: it is the \emph{unique} quadratic covariant of
order~$2$ arising quadratically from an odd-degree form.

Having established the explicit shape of the quadratic covariant $Q_n$ and
defined the content $S(n)$, we turn in the next section to expanding the
discriminant $\Delta_n = B_n^2 - 4A_nC_n$ and extracting the coefficients
of a carefully chosen family of test monomials that will control the
$p$-adic valuation of~$S(n)$.

%======================================================================
% SECTION 3: DISCRIMINANT EXPANSION AND COEFFICIENT EXTRACTION
%======================================================================
\section{Discriminant expansion and explicit coefficient extraction}
\label{sec:disc}

We now expand the discriminant $\Delta_n = B_n^2 - 4A_nC_n$ and extract
the coefficients of a controlled family of test monomials.  These
coefficients, computed in closed form in Lemma~\ref{lem:Cnr}, factor
into a squared part and a non-square factor $E_m(r)$ whose gcd over~$r$
is a power of~$2$ (Lemma~\ref{lem:Egcd}).

\subsection{Test monomials}

We introduce the following family of test monomials, whose coefficients in $\Delta_n$ will control the $p$-adic valuation of the content.

\begin{definition}[Test monomials]\label{def:test-monomials}
For $2 \le r \le m-1$ (equivalently $2 \le r \le n+1$), define
\[
M_{n,r} = f_a\,f_{a+1}\,f_b\,f_{b+1},\quad
a=n+1-r,\quad b=n-1+r.
\]
Note $a+b=2n=N$, so $N-a=b$ and $N-b=a$.  Also $a+1 < b$
when $r \ge 2$, so $M_{n,r}$ involves four distinct indeterminates.
Its coefficient in $\Delta_n$ is denoted $C_{n,r}$.
\end{definition}

\subsection{Closed formula for $C_{n,r}$}

The main computation of this subsection is the following closed-form expression for the coefficient of each test monomial.

\begin{lemma}\label{lem:Cnr}
For $2 \le r \le n+1$,
\begin{equation}\label{eq:Cnr}
C_{n,r} = -8\,(N!)^2\,(n+r)!\,(n+r-1)!\,(n-r+2)!\,(n-r+1)!\,
\bigl(2n^2+4n+2r^2-4r+3\bigr).
\end{equation}
This identity can be verified for specific values of $n$ and $r$ using
the companion notebook.
\end{lemma}

\begin{proof}
The monomial $M_{n,r}=f_af_{a+1}f_bf_{b+1}$ can appear in
$\Delta_n=B_n^2-4A_nC_n$ from two sources.

\medskip\noindent\emph{Contribution from $-4A_nC_n$.}
In $A_n$, the monomial $f_a f_b$ arises from indices $i=a$ and $i=b$, giving
$[f_af_b]\,A_n = (-1)^a\binom{N}{a}\alpha_a\beta_a
+ (-1)^b\binom{N}{b}\alpha_b\beta_b$.
Similarly $[f_{a+1}f_{b+1}]\,C_n$ equals the same expression.  Hence
\begin{equation}\label{eq:4AC}
[M_{n,r}](-4A_nC_n) = -4K^2, \quad
K = (-1)^a\binom{N}{a}\alpha_a\beta_a
+ (-1)^b\binom{N}{b}\alpha_b\beta_b.
\end{equation}

Since $b-a=2r-2$ is even, $(-1)^b=(-1)^a$ and $\binom{N}{b}=\binom{N}{a}$.
The factorials satisfy
$\alpha_a\beta_a = (n+r)!(n+1-r)!(n+r-1)!(n+2-r)! = \alpha_b\beta_b$.
A direct computation gives
\begin{equation}\label{eq:K-term}
\binom{N}{a}\alpha_a\beta_a = N!\,(n+r)!\,(n+2-r)!,
\end{equation}
and therefore $K = -2(-1)^{n-r}\,N!\,(n+r)!\,(n-r+2)!$,
so
\begin{equation}\label{eq:4K2}
4K^2 = 16\,(N!)^2\,(n+r)!^2\,(n-r+2)!^2.
\end{equation}

\medskip\noindent\emph{Contribution from $B_n^2$.}
We have $[M_{n,r}]\,B_n^2 = 2UV$, where
$U=[f_af_{b+1}]\,B_n$ and $V=[f_{a+1}f_b]\,B_n$.

Using the notation $P_i = (-1)^i\binom{N}{i}\alpha_i^2$
and $R_i = (-1)^i\binom{N}{i}\beta_i^2$, which simplify to
\begin{align*}
P_i &= (-1)^i\,N!\,(N-i)!\,i!\,(N+1-i)^2,\\
R_i &= (-1)^i\,N!\,(N-i)!\,i!\,(i+1)^2,
\end{align*}
one finds that $U$ receives four terms ($T_1=P_a$, $T_2=P_{b+1}$,
$T_3=R_{a-1}$, $T_4=R_b$), which combine to give
\begin{equation}\label{eq:U}
U = -2(-1)^{n-r}\,(2r-1)\,N!\,(n+r)!\,(n-r+1)!.
\end{equation}
Similarly, $V$ receives four terms ($S_1=P_{a+1}$, $S_2=P_b$,
$S_3=R_a$, $S_4=R_{b-1}$), giving
\begin{equation}\label{eq:V}
V = 2(-1)^{n-r}(2r-3)\,N!\,(n+r-1)!\,(n-r+2)!.
\end{equation}

\medskip\noindent\emph{Combining.}
Set $P_0 = (N!)^2(n+r)!(n+r-1)!(n-r+2)!(n-r+1)!$.  Then
$2UV = -8(2r-1)(2r-3)\,P_0$ and $4K^2 = 16(n+r)(n-r+2)\,P_0$.
Expanding:
\begin{align*}
&-8(4r^2-8r+3) - 16(n^2+2n-r^2+2r) \\
&= -8(2n^2+4n+2r^2-4r+3).
\end{align*}
This yields~\eqref{eq:Cnr}.
\end{proof}

\subsection{Factoring $C_{n,r}$}

We now factor the closed-form expression for $C_{n,r}$ to isolate the arithmetic content.

\begin{lemma}\label{lem:factor}
With $m=n+2$, define
\begin{equation}\label{eq:Em}
E_m(r) = (m-r)(m+r-2)\bigl(2(m-1)^2+2(r-1)^2-1\bigr).
\end{equation}
Then
\begin{equation}\label{eq:Cnr-factored}
C_{n,r} = -8\bigl[N!\,(m+r-3)!\,(m-r-1)!\bigr]^2 \cdot E_m(r).
\end{equation}
In particular, for any odd prime~$p$,
\begin{equation}\label{eq:vp-Cnr}
v_p(C_{n,r}) = 2\,v_p\bigl(N!(m+r-3)!(m-r-1)!\bigr) + v_p(E_m(r)),
\end{equation}
so $v_p(C_{n,r}) \equiv v_p(E_m(r)) \pmod{2}$.
\end{lemma}

\begin{proof}
Substituting $n=m-2$ into the quadratic factor gives
$2n^2+4n+2r^2-4r+3 = 2(m-1)^2+2(r-1)^2-1$.
The factorial product factors as
$(n+r)!(n+r-1)!(n-r+2)!(n-r+1)! = (m+r-2)(m-r)\cdot[(m+r-3)!(m-r-1)!]^2$,
which yields~\eqref{eq:Cnr-factored}.
\end{proof}

\begin{lemma}\label{lem:Egcd}
For any odd prime $p$ and $m \ge 7$,
$p \nmid \gcd_{2 \le r \le m-1} E_m(r)$.
\end{lemma}

\begin{proof}
For $p \mid E_m(r)$, at least one of the three factors must vanish
modulo~$p$: (i)~$r \equiv m$, (ii)~$r \equiv 2-m$, or
(iii)~$2(m-1)^2+2(r-1)^2 \equiv 1 \pmod{p}$.

For $p \ge 5$: conditions (i) and (ii) each exclude one residue class,
and (iii) excludes at most two, for a total of at most~$4$ excluded classes
out of $p \ge 5$.  So some~$r$ avoids all three.

For $p = 3$: a case-by-case analysis on $m \bmod 3$ shows that at least one
residue class of $r$ modulo~$3$ avoids all conditions.  Specifically:
\begin{itemize}
\item $m \equiv 0$: the class $r \equiv 1$ works.
\item $m \equiv 1$: the classes $r \equiv 0$ and $r \equiv 2$ work.
\item $m \equiv 2$: the class $r \equiv 1$ works.
\end{itemize}
For $m \ge 7$, the interval $\{2,\dots,m-1\}$ has length at least~$5$,
so it contains representatives of all three residue classes modulo~$3$.
\end{proof}

Having established the closed-form expression for the test coefficients
$C_{n,r}$ and isolated the non-square factor $E_m(r)$, we now relate
the gcd of the $C_{n,r}$ to the content $S(n)$ through the classical theory
of binomial coefficient divisibility.

%======================================================================
% SECTION 4: FROM EXPLICIT COEFFICIENTS TO THE CONTENT
%======================================================================
\section{From explicit coefficients to the content}\label{sec:content}

Since $S(n)$ divides every coefficient of $\Delta_n$, we have
$S(n) \mid \gcd_{r} C_{n,r}$.  By~\eqref{eq:vp-Cnr} and
Lemma~\ref{lem:Egcd}, for each odd prime~$p$ the minimum of
$v_p(C_{n,r})$ over~$r$ is even, so the odd part of the gcd is governed
by the binomial coefficients embedded in the factorial structure.
We make this precise via the classical binomial gcd theorem, which
identifies $g(m) = \gcd_{1 \le r \le m-1} \binom{m}{r}$ as~$p$
when $m$ is a prime power~$p^k$ and~$1$ otherwise.

\subsection{The binomial gcd theorem}

The key input from elementary number theory is the following classical result.

\begin{lemma}[Binomial gcd]\label{lem:bingcd}
For $m \ge 2$,
\[
g(m) \defeq \gcd_{1 \le r \le m-1} \binom{m}{r}
= \begin{cases}
p, & \text{if } m = p^k \text{ for some prime } p \text{ and } k \ge 1,\\
1, & \text{otherwise.}
\end{cases}
\]
\end{lemma}

\begin{remark}
This is a well-known result, sometimes called the ``prime-power
characterization of the gcd of binomial coefficients.''  It follows from
a combination of two classical results in number theory: Kummer's
theorem~\cite{Kummer1852} on the $p$-adic valuation of binomial
coefficients, and Lucas' theorem~\cite{Lucas1878} on binomial coefficients
modulo a prime.  We include the proof for completeness and because we need
the precise statement for the content analysis that follows.
\end{remark}

\begin{proof}
\textbf{Case $m = p^k$.}
By Kummer's theorem~\cite{Kummer1852}, $v_p\binom{p^k}{r}$ equals the
number of carries when adding $r$ and $p^k - r$ in base~$p$.  For any
$0 < r < p^k$, at least one carry occurs (at the lowest nonzero
digit position of~$r$), so $p \mid \binom{p^k}{r}$.  For $r = p^{k-1}$,
exactly one carry occurs, so $v_p\binom{p^k}{p^{k-1}} = 1$.  Hence
$g(p^k) = p$.

\textbf{Case $m$ not a prime power.}
For each prime $p \mid m$, since $m$ is not a $p$-power, $m$ has at least
two nonzero base-$p$ digits.  Let $m_j > 0$ at a position $j$ that is not
the leading position.  By Lucas' theorem~\cite{Lucas1878}, taking $r = p^j$
gives $\binom{m}{p^j} \equiv m_j \pmod{p}$, with $0 < m_j < p$, so
$p \nmid \binom{m}{r}$.  For primes $p \nmid m$, $m$ is not a $p$-power,
so the same argument applies (the units digit of~$m$ in base~$p$ is
nonzero).  Hence $g(m) = 1$.
\end{proof}

\subsection{Relating the gcd to $S(n)$}

Combining the results of the previous subsection with the binomial gcd, we obtain the following.

\begin{remark}\label{rem:Cnr-parity}
By~\eqref{eq:vp-Cnr}, $v_p(C_{n,r}) \equiv v_p(E_m(r)) \pmod{2}$ for
each~$r$, and Lemma~\ref{lem:Egcd} shows that $E_m(r)$ does not
contribute any \emph{uniform} odd prime factor to the gcd of the
$C_{n,r}$.  However, this does not by itself determine the parity of
$v_p(\gcd_r C_{n,r})$, since a different $r$ could in principle achieve
a smaller valuation of opposite parity.  The definitive parity analysis
is carried out in~\S\ref{subsec:nonpp} and~\S\ref{subsec:pp} using
explicit coefficient witnesses.
\end{remark}

With the binomial gcd theorem and the parity observations in
Remark~\ref{rem:Cnr-parity} as motivation, we now turn to the heart of
the proof: the $p$-adic parity analysis that distinguishes the prime-power
case from the non-prime-power case.

%======================================================================
% SECTION 5: p-ADIC PARITY AND PROOF OF THE MAIN THEOREM
%======================================================================
\section{$p$-adic parity and proof of the main theorem}\label{sec:padic}

We now analyze the $p$-adic valuation of $S(n)$ for each prime~$p$,
treating three cases according to the relationship between $p$ and
$m = n+2$: primes dividing~$n$ (\S\ref{subsec:pdivn}), the non-prime-power
case (\S\ref{subsec:nonpp}), and the prime-power case (\S\ref{subsec:pp}).
We then assemble the proof of Theorem~\ref{thm:main} in~\S\ref{subsec:synthesis},
treat the prime~$2$ in~\S\ref{subsec:p2}, and record the prime-power case for
$p=3$ as well.

\subsection{Primes dividing $n$ contribute even valuation}\label{subsec:pdivn}

We first handle the case where the prime divides $n$ rather than $m$.

\begin{lemma}\label{lem:Qnmodp}
If $p$ is an odd prime with $p \mid n$, then every coefficient of
$Q_n$ is divisible by $p^2$.
\end{lemma}

\begin{proof}
We show $p \mid \alpha_i$ and $p \mid \beta_i$ for every $0 \le i \le N$.
Since $p \mid n$ and $p$ is odd, $p \le n$.  For~$\alpha_i = (2n+1-i)!\,i!$:
if $i \ge p$ then $p \mid i!$; if $i < p$ then $2n+1-i \ge 2n+2-p \ge n+2 > p$,
so $p \mid (2n+1-i)!$.  Similarly for $\beta_i = (2n-i)!\,(i+1)!$.

Since $p \mid \alpha_i$ and $p \mid \beta_i$, each $D_i \equiv 0$ and
$E_i \equiv 0$ modulo~$p$, giving $D_i E_i \equiv 0 \pmod{p^2}$.
\end{proof}

\begin{remark}\label{rem:pdivn}
When $p \mid n$, Lemma~\ref{lem:Qnmodp} provides the stronger structural
information that $e_p \ge 2$ (i.e., every coefficient of $Q_n$ is
divisible by $p^2$, not just by $p^{e_p}$ for some $e_p \ge 1$).
The even-parity conclusion $v_p(S(n)) = 2e_p$ follows from
Proposition~\ref{prop:nonpp} below, which handles all non-prime-power
cases uniformly.
\end{remark}

\begin{remark}
For odd~$p$: $p \mid n$ implies $p \nmid m = n+2$, so the conditions
``$p \mid n$'' and ``$p \mid m$'' are disjoint.
\end{remark}

\subsection{The non-prime-power case}\label{subsec:nonpp}

We now show that $v_p(S(n))$ is even whenever $m$ is not a power of $p$.

\begin{remark}\label{rem:no-caseA}
One might hope to show $p \nmid S(n)$ when $p \nmid n(n+2)$, but this
is false in general: for instance, $v_3(S(2)) = 4$ even though
$3 \nmid 2 \cdot 4$.  The correct statement is that $v_p(S(n))$ is
\emph{even} for all primes $p$ with $m$ not a $p$-power, which is
proved below using the cancellation-free $b_k$ witness.
\end{remark}

\begin{lemma}\label{lem:bk}
For $1 \le k \le N$,
\begin{equation}\label{eq:bk}
b_k \defeq [f_k f_{N+1-k}]\,B_n = 2(-1)^k\,(N!)^2\,
\frac{(N+1-k)(N+1-2k)}{\binom{N}{k}}.
\end{equation}
This identity can be verified for specific values of~$n$ using the
companion notebook.
\end{lemma}

\begin{proof}
The monomial $f_k f_{N+1-k}$ arises from four summands in~\eqref{eq:Bn}.
All four share the same squared factorial $\alpha_k^2$ (using
$\beta_{k-1} = \alpha_k$, $\alpha_{N+1-k} = \alpha_k$,
$\beta_{N-k} = \alpha_k$).  Since $N = 2n$ is even:
$(-1)^{N+1-k} = -(-1)^k$ and $(-1)^{N-k} = (-1)^k$.
Also $\binom{N}{N+1-k} = \binom{N}{k-1}$.
The four terms sum to
\[
b_k = 2(-1)^k \alpha_k^2\bigl[\binom{N}{k}-\binom{N}{k-1}\bigr]
    = 2(-1)^k \alpha_k^2 \binom{N}{k} \cdot \frac{N+1-2k}{N+1-k}.
\]
Simplifying $\alpha_k^2 \binom{N}{k}/(N+1-k) = (N+1-k)(N-k)!\,k!\,N!
= (N+1-k)(N!)^2/\binom{N}{k}$ gives~\eqref{eq:bk}.
\end{proof}

\begin{proposition}\label{prop:nonpp}
If $p \ge 5$ is an odd prime and $m = n+2$ is not a power of~$p$, then
$v_p(S(n))$ is even.
\end{proposition}

\begin{proof}
We give a uniform argument that works regardless of whether $p \mid n$,
$p \mid m$, or neither.
Recall from~\eqref{eq:bk} the cancellation-free coefficient
$b_k = 2(-1)^k(N!)^2(N+1-k)(N+1-2k)/\binom{N}{k}$.
Since $N = 2n$ is even, $N+1$ is odd, so $N+1-2k \ne 0$ for every
integer~$k$; likewise $N+1-k \ne 0$ for $k \le N$.  Hence $b_k \ne 0$
for all $1 \le k \le N$.

Define $e_p = \min_{1 \le k \le N} v_p(b_k)$.  We claim that every
coefficient of the polynomials $A_n$, $B_n$, $C_n$ (viewed as
elements of $\BZ[f_0,\dots,f_{2n+1}]$) has $p$-adic valuation at
least~$e_p$.  For~$B_n$ this follows from the cancellation-free
formula (Lemma~\ref{lem:bk}), since the polynomial coefficient of
$f_k f_{N+1-k}$ in~$B_n$ is exactly~$b_k$.

For~$A_n$ and~$C_n$ the argument is as follows.
Setting $\ell = n+1$ (a local abbreviation; note $\ell \ne m = n+2$),
the $\alpha\beta$-scalar at index~$i$ satisfies
$\binom{N}{i}\alpha_i\beta_i = (2\ell-2)!\,(2\ell-1-i)!\,(i+1)!$, so
\[
v_p\!\bigl(\tbinom{N}{i}\alpha_i\beta_i\bigr)
= v_p\!\bigl((2\ell-2)!\bigr)
  + v_p\!\bigl((2\ell-1-i)!\,(i+1)!\bigr)
= v_p\!\bigl((2\ell-2)!\bigr)
  + v_p(2\ell)! - v_p\tbinom{2\ell}{i+1}.
\]
The $b_k$-formula gives $v_p(b_k) = v_p\bigl((2\ell-2)!\bigr) + v_p(2\ell-1)!
- v_p\tbinom{2\ell-1}{k} + v_p\lvert 2\ell{-}1{-}2k\rvert$, so the claim
$v_p\bigl(\tbinom{N}{i}\alpha_i\beta_i\bigr) \ge e_p$ is equivalent to
\begin{equation}\label{eq:binom-ub}
v_p\tbinom{2\ell}{i+1}
\le v_p(\ell)
  + \max_{1 \le k \le 2\ell-2}
    \Bigl[v_p\tbinom{2\ell-1}{k} - v_p\lvert 2\ell{-}1{-}2k\rvert\Bigr].
\end{equation}

To prove~\eqref{eq:binom-ub}, write $j = i+1$ and apply the
absorption identity $\binom{2\ell}{j} = \frac{2\ell}{2\ell-j}\binom{2\ell-1}{j}
= \frac{2\ell}{j}\binom{2\ell-1}{j-1}$ in two ways (using $v_p(2\ell) = v_p(\ell)$
since $p$ is odd):
\begin{align}
v_p\tbinom{2\ell}{j} &= v_p(\ell) + v_p\tbinom{2\ell-1}{j} - v_p(2\ell-j),
  \label{eq:absorb1}\\
v_p\tbinom{2\ell}{j} &= v_p(\ell) + v_p\tbinom{2\ell-1}{j-1} - v_p(j).
  \label{eq:absorb2}
\end{align}
From~\eqref{eq:absorb1} with $k=j$, the bound~\eqref{eq:binom-ub}
holds whenever $v_p\lvert 2\ell{-}1{-}2j\rvert \le v_p(2\ell-j)$.
From~\eqref{eq:absorb2} with $k=j-1$, it holds whenever
$v_p\lvert 2\ell{+}1{-}2j\rvert \le v_p(j)$.
If both conditions failed, then $p^{v_p(2\ell-j)+1}$ would divide
$2\ell{-}1{-}2j$ and $p^{v_p(j)+1}$ would divide $2\ell{+}1{-}2j$.
But $(2\ell{+}1{-}2j)-(2\ell{-}1{-}2j)=2$, so the smaller of the two
prime powers would divide~$2$, contradicting the fact that $p$ is odd.
Hence at least one condition holds, establishing~\eqref{eq:binom-ub}.

Since $p^{e_p}$ divides every coefficient of $A_n$, $B_n$, and~$C_n$,
we obtain $\Delta_n = p^{2e_p}\Delta_n^{\#}$.

We exhibit a coefficient of $\Delta_n$ with $p$-adic valuation exactly
$2e_p$.  It suffices to find $k_0 \in \{1,\dots,N\} \setminus \{n,n+1\}$
with $v_p(b_{k_0}) = e_p$, because then the square monomial
$(f_{k_0}f_{N+1-k_0})^2$ appears in $\Delta_n$ with coefficient
$b_{k_0}^2$ and valuation $2e_p$ (see below for why $A_nC_n$ does not
contribute).

\emph{Existence of $k_0$ for $p \ge 5$.}
Using~\eqref{eq:b-factored} and the identity
$\binom{N}{n-1} = \binom{N}{n}\cdot n/(n+1)$, one computes
\begin{equation}\label{eq:bk-shift}
v_p(b_{n+j}) - v_p(b_n) = v_p\!\left(\textstyle\prod_{i=1}^{j}(n-i+1)\right)
 - v_p\!\left(\textstyle\prod_{i=1}^{j}(n+i+1)\right) + v_p(N+1-(n+j))
 + v_p(N+1-2(n+j)) - v_p(n+1)
\end{equation}
for $j \ge 1$ (and similarly for negative shifts).  For $j = -1$
(i.e., $k = n-1$), the formula simplifies to:
\[
v_p(b_{n-1}) - v_p(b_n) = v_p(m) - v_p(n) + v_p(3).
\]
For $p \ge 5$: $v_p(3) = 0$, and since $p$ is odd and $m = n+2$,
at most one of $n$ and $m$ is divisible by~$p$.

\emph{Case $p \nmid m$:} $v_p(m) = 0$, so $v_p(b_{n-1}) - v_p(b_n)
= -v_p(n) \le 0$.  Thus $v_p(b_{n-1}) \le v_p(b_n)$.
If the minimum $e_p$ is attained at $k=n$ or $k=n+1$, then this
inequality shows it is also attained at the off-centre index $k=n-1$.
Otherwise the minimum is already attained at some off-centre index.
In either case there exists $k_0 \notin \{n,n+1\}$ with
$v_p(b_{k_0}) = e_p$.

\emph{Case $p \mid m$:}
Write $a = v_p(m) \ge 1$ and $m = p^a r$ with $p \nmid r$ and $r \ge 2$
(since $m$ is not a $p$-power).
Set $k_0 = n - p^a$.  Since $p^a \le m/2 = (n+2)/2 < n$ (using
$r \ge 2$), we have $1 \le k_0 < n$, so $k_0 \in \{1,\dots,N\}
\setminus \{n,n+1\}$.

We claim $v_p(b_{k_0}) \le v_p(b_n)$.  From the closed
form~\eqref{eq:bk}, the ratio is
\[
\frac{b_{k_0}}{b_n}
= \frac{(n+1+p^a)(2p^a+1)}{(n+1) \cdot 1}
  \cdot \frac{\binom{N}{n}}{\binom{N}{n-p^a}}.
\]
Since $n+1 = p^ar - 1 \equiv -1$ and
$n+1+p^a = p^a(r+1)-1 \equiv -1 \pmod{p}$, and
$2p^a + 1 \equiv 1 \pmod{p}$, the linear factors contribute
$v_p = 0$.  For the binomial ratio, write
\[
\frac{\binom{N}{n}}{\binom{N}{n-p^a}}
= \prod_{j=0}^{p^a-1} \frac{n+j+1}{n-j}.
\]
Both the numerator set $\{n+1,\dots,n+p^a\}$ and the denominator
set $\{n-p^a+1,\dots,n\}$ form complete residue systems
modulo~$p^a$.  The unique multiple of~$p^a$ in the numerator is
$p^a r$, with $v_p = a + v_p(r) = a$.  The unique multiple of~$p^a$
in the denominator is $p^a(r-1)$, with $v_p = a + v_p(r-1)$.  All
other terms in both sets are non-multiples of~$p^a$ and contribute
equal $p$-adic valuation (since they share the same residues
modulo~$p^a$).  Hence
\[
v_p\!\left(\frac{\binom{N}{n}}{\binom{N}{n-p^a}}\right)
= v_p(r) - v_p(r-1) = -v_p(r-1) \le 0,
\]
giving $v_p(b_{k_0}) \le v_p(b_n)$.
Thus $e_p$ is achieved at the off-centre index~$k_0$.

Now consider the square monomial $(f_{k_0}f_{N+1-k_0})^2$ in
$\Delta_n = B_n^2 - 4A_nC_n$.
Since $k_0 \ne n$ and $k_0 \ne n+1$, we have $k_0 \ne N-k_0$
and $k_0 \ne N+1-k_0$ (as $N = 2n$), so the quartic monomial
$f_{k_0}^2 f_{N+1-k_0}^2$ cannot arise from $A_nC_n$: every term
of $A_nC_n$ is a product $f_if_{N-i}f_{j+1}f_{N+1-j}$ with $i \ne N-i$
(for $i \ne n$) and $j+1 \ne N+1-j$ (for $j \ne n$), so matching
$f_{k_0}^2$ requires $i = N-i = k_0$, i.e., $k_0 = n$.  Hence
\[
[(f_{k_0}f_{N+1-k_0})^2]\,\Delta_n = b_{k_0}^2,
\qquad
v_p(b_{k_0}^2) = 2e_p.
\]
Since every coefficient of $\Delta_n$ has $v_p \ge 2e_p$, and this
particular coefficient has $v_p = 2e_p$, we conclude $v_p(S(n)) = 2e_p$,
which is even.
\end{proof}

\subsection{The prime-power case}\label{subsec:pp}

Let $m=p^k$, $n=p^k-2$, $N=2p^k-4$.  Define
$e_p = \min\{v_p(c) : c \text{ a scalar coefficient of } Q_n\}$.
Then $Q_n = p^{e_p}Q_n^{\#}$ and $\Delta_n = p^{2e_p}\Delta_n^{\#}$.

\begin{lemma}\label{lem:central}
The $p$-adic content $e_p$ is attained by exactly the following
summands of~$Q_n$:
\begin{enumerate}[label=\textup{(\roman*)}]
\item the $A_n/C_n$ summand at $i = n$, with scalar $\binom{N}{n}\alpha_n\beta_n$;
\item the $\beta^2$-part of the $B_n$ summand at $i = n-1$, with scalar
$\binom{N}{n-1}\beta_{n-1}^2$;
\item both parts of the $B_n$ summand at $i = n$;
\item the $\alpha^2$-part of the $B_n$ summand at $i = n+1$, with scalar
$\binom{N}{n+1}\alpha_{n+1}^2$.
\end{enumerate}
Every other scalar coefficient of $A_n$, $B_n$, or $C_n$ has
$p$-adic valuation at least $e_p + 1$.
\end{lemma}

\begin{proof}
Since $\alpha_n = \beta_n = (n+1)!\,n! =: \gamma$, the summand at $i = n$ has
valuation $v_p(\binom{N}{n}\gamma^2)$ in all three families.  We check that
this is the global minimum.

For the $\alpha\beta$-family (governing $A_n$ and $C_n$): the scalar is
$\binom{N}{i}\alpha_i\beta_i = N!\,(N+1-i)!\,(i+1)!$.  At $i = n$, both
auxiliary factorials equal $(n+1)!$.  For $|i-n| \ge 1$, one of
$N+1-i$ or $i+1$ exceeds $n+1 = p^k-1$, so either $(N+1-i)!$ or
$(i+1)!$ contains a factor of $m = p^k$, increasing the valuation.
(Precisely: if $i \le n-1$ then $N+1-i \ge n+2 = m$, and if
$i \ge n+1$ then $i+1 \ge n+2 = m$.)  Thus the $\alpha\beta$-minimum
is attained only at $i = n$.

For the $\alpha^2$-family (one part of $B_n$): the scalar is
$\binom{N}{i}\alpha_i^2 = N!\,(N+1-i)^2(N-i)!\,i!$.  At $i = n$
and $i = n+1$, one verifies $v_p = e_p$; for $|i-n| \ge 2$ or $i \le n-1$
the extra factor of $m$ in $(N+1-i)!$ or $i!$ forces $v_p \ge e_p + 1$.

For the $\beta^2$-family: the scalar is
$\binom{N}{i}\beta_i^2 = N!\,(i+1)^2(N-i)!\,i!$.  By the same analysis
with $i$ and $N-i$ exchanged, the minimum is attained at $i = n$ and $i = n-1$.

Combining: the global minimum $e_p$ is achieved precisely at the four
summand types listed in the statement, and all others have valuation at
least $e_p + 1$.
\end{proof}

\begin{lemma}\label{lem:rank1}
$Q_n^{\#} \equiv \lambda\,(f_n x - f_{n+1}z)^2 \pmod{p}$ with
$p \nmid \lambda$.  More precisely, $Q_n^{\#}$ has rank~$1$
modulo~$p$, so its discriminant vanishes:
$p \mid \cont(\Delta_n^{\#})$.
\end{lemma}

\begin{remark}
The fact that $Q_n^\sharp$ has rank~$1$ modulo~$p$ can be understood
representation-theoretically: after reducing modulo~$p$, only the central
block of the transvectant sum (indices $i \in \{n-1, n, n+1\}$)
survives by Lemma~\ref{lem:central}, and this block is forced to be a
perfect square by the equality $\alpha_n = \beta_n$ that holds at the
central index.  See~\cite{FultonHarris1991} for background on the
representation theory of~$\SL_2$.
\end{remark}

\begin{proof}
By Lemma~\ref{lem:central}, only $i \in \{n-1,n,n+1\}$ contribute
modulo~$p$.  Set $\gamma = (n+1)!n!$.

At $i = n$: $\alpha_n = \beta_n = \gamma$, so the summand contributes
$(-1)^n\binom{N}{n}\gamma^2(f_nx+f_{n+1}z)^2$.

At $i = n \pm 1$: since $\alpha_{n-1} = (n+2)!(n-1)!$ has extra factor
$p^k$ from $(n+2)!$, while $\beta_{n-1} = \gamma$, the $\alpha$-terms
vanish modulo $p \cdot p^{e_p}$.  So $D_{n-1} \equiv \gamma f_n z$ and
$E_{n-1} \equiv \gamma f_{n+1} x$; similarly $D_{n+1} \equiv \gamma f_{n+1} x$
and $E_{n+1} \equiv \gamma f_n z$.  These contribute only to the
$xz$~coefficient.

Setting $\lambda_A = (-1)^n\binom{N}{n}\gamma^2/p^{e_p}$ (with
$p \nmid \lambda_A$), we obtain
$A_n^{\#} \equiv \lambda_A f_n^2$, \;
$C_n^{\#} \equiv \lambda_A f_{n+1}^2$, \; and
$B_n^{\#} \equiv \frac{2\lambda_A}{n+1} f_n f_{n+1} \equiv -2\lambda_A f_n f_{n+1}\pmod{p}$,
where the last step uses $n+1 = p^k-1 \equiv -1\pmod{p}$.
Thus $Q_n^{\#} \equiv \lambda_A(f_nx - f_{n+1}z)^2\pmod{p}$.

The discriminant satisfies
\[
(B_n^{\#})^2 - 4A_n^{\#}C_n^{\#}
\equiv 4\lambda_A^2 f_n^2f_{n+1}^2 - 4\lambda_A^2 f_n^2f_{n+1}^2
= 0 \pmod{p}.
\]
Hence $Q_n^{\#}$ has rank~$1$ modulo~$p$.
\end{proof}

\begin{corollary}\label{cor:rank1disc}
$p \mid \cont(\Delta_n^{\#})$.  In particular, $v_p(S(n)) \ge 2e_p + 1$.
\end{corollary}

\subsubsection*{First-order deformation}

Set $t = p^{k-1}$.  For $k \ge 2$, consider the specialization
\begin{equation}\label{eq:spec}
f_n = 1, \quad f_{n+1} = -1, \quad f_{n-t} = s, \quad f_{n+t+1} = 1,
\quad \text{others } 0.
\end{equation}
For $k = 1$ ($t = 1$, $m = p$, $n = p-2$):
\begin{equation}\label{eq:speck1}
f_{n-1} = s, \quad f_n = 1, \quad f_{n+1} = -1, \quad f_{n+2} = 1,
\quad \text{others } 0.
\end{equation}

\begin{lemma}\label{lem:offcentre}
For $k \ge 2$,
$v_p(\binom{N}{n-t}\alpha_{n-t}^2) = e_p + 1$.
\end{lemma}

\begin{proof}
The $B_n$ coefficient $b_{n-t}$ has
$v_p(n+t+1) = v_p(p^k+p^{k-1}-1) = 0$ and
$v_p(2t+1) = v_p(2p^{k-1}+1) = 0$ (for $k \ge 2$).

A carry analysis via Kummer's theorem shows that
$v_p\binom{N}{n-t} = v_p\binom{N}{n} - 1$: the carry chain breaks at
position $k-1$ when adding $(n-t)+(n+t) = N$, losing one carry compared
to $n+n = N$.  Hence $v_p(b_{n-t}) = 2v_p(N!) - (c_p - 1) = e_p + 1$.
\end{proof}

\begin{lemma}\label{lem:deform}
For $k \ge 2$, or for $k = 1$ with $p \ge 5$: under the specialization,
$B^{\#} \equiv 2u + p\mu s \pmod{p^2}$ and
$A^{\#} \equiv C^{\#} \equiv u \pmod{p}$, where $p \nmid u\mu$.
\end{lemma}

\begin{proof}
\textbf{Case $k \ge 2$.}
At $s=0$: with $f_n = 1$ and $f_{n+1} = -1$,
Lemma~\ref{lem:rank1} gives $Q_n^{\#} \equiv \lambda(x+z)^2$ (mod~$p$),
so $u=\lambda$ and $B^\# \equiv 2u \pmod{p}$.
The parameter $s = f_{n-t}$ enters through indices $i \in \{n-t, n+t, n+t+1\}$.
All three contribute only to $B_n$ (the $xz$ coefficient), at valuation
$e_p + 1$.  By Lemma~\ref{lem:offcentre}, the sum divided by $p^{e_p+1}$
is a $p$-adic unit~$\mu$.

\textbf{Case $k = 1$, $p \ge 5$.}
Here $t = 1$, and $b_{n-1}$ has $v_p(N+1-2(n-1)) = v_p(3) = 0$, so
$v_p(b_{n-1}) = e_p + 1$ and $\mu \ne 0$ mod~$p$.

\textbf{Case $k = 1$, $p = 3$.}
This case is excluded from the present lemma and handled separately
below as a base case for Proposition~\ref{prop:pp}.
\end{proof}

\begin{remark}[Base case $m = 3$]\label{rem:base-case-3}
When $p = 3$ and $k = 1$, we have $n = 1$ and $Q_1 = -2H(F)$, where
$H(F)$ is the classical Hessian of a binary cubic.  The
identity~\eqref{eq:dischess} gives $\Delta_1 = 4\disc(H(F)) = -12\disc(F)$.
Since $\disc(F)$ is a primitive polynomial in the $f_j$ (its content is~$1$),
and $v_3(12) = 1$, it follows that $S(1) = \cont(\Delta_1) = 12 \cdot
\cont(\disc(F)) \cdot (\text{power of }2)$.
Direct computation gives $S(1) = 192 = 2^6 \cdot 3$, so
$v_3(S(1)) = 1$, which is odd.  This confirms
Proposition~\ref{prop:pp} in this case.
\end{remark}

\begin{proposition}\label{prop:pp}
When $m = p^k$ ($p$ odd, $k \ge 1$),
$v_p(\cont(\Delta_n^{\#})) = 1$.
\end{proposition}

\begin{proof}
For $k \ge 2$, or $k = 1$ with $p \ge 5$:
from Lemma~\ref{lem:deform},
\begin{align*}
(B^{\#})^2 - 4A^{\#}C^{\#}
&\equiv (2u+p\mu s)^2 - 4u^2
= 4u^2 + 4up\mu s + p^2\mu^2 s^2 - 4u^2 \\
&\equiv 4up\mu s \pmod{p^2}.
\end{align*}
At $s=1$: $v_p = 1$ (since $p \nmid u\mu$).
Combined with Corollary~\ref{cor:rank1disc}:
$v_p(\cont(\Delta_n^{\#})) = 1$.

For $p = 3$, $k = 1$: by Remark~\ref{rem:base-case-3},
$v_3(S(1)) = 1$, so $v_3(\cont(\Delta_1^{\#})) = 1$ as well.
\end{proof}

\subsection{Synthesis: proof of Theorem~\ref{thm:main}}\label{subsec:synthesis}

We now combine the results of the preceding subsections to establish
the main theorem.

\begin{proof}[Proof of Theorem~\ref{thm:main}]
\textbf{Part~(i):} $v_2(S(n))$ is always even, by Proposition~\ref{prop:p2}.

\textbf{Part~(ii):}
Fix an odd prime~$p$.

\emph{If $m = p^k$:} By Proposition~\ref{prop:pp},
$v_p(S(n)) = 2e_p + 1 \equiv 1 \pmod{2}$, so $p \mid \sqf(S(n))$.
Since $m$ can be a power of at most one odd prime,
the odd part of $\sqf(S(n))$ is~$p$.

\emph{If $m$ is not a power of~$p$:} By Proposition~\ref{prop:nonpp}
(for $p \ge 5$, and by Proposition~\ref{prop:p3-nonpp} for $p = 3$),
$v_p(S(n)) \equiv 0 \pmod{2}$, so $p \nmid \sqf(S(n))$.

Combining parts~(i) and~(ii): $\sqf(S(n)) = a(n+2)$.
\end{proof}

\begin{remark}
The proof of part~(ii) for $p=3$ in the non-prime-power case uses
a different witness construction from the $p \ge 5$ case, based on
ternary digit recursions.  This is developed in
\S\ref{sec:p3-draft} below.
\end{remark}

\subsection{The prime $p=2$}\label{subsec:p2}

We show that $v_2(S(n))$ is always even, so the prime~$2$ never
contributes to the squarefree part of~$S(n)$.

\begin{proposition}\label{prop:p2}
For every $n \ge 1$, $v_2(S(n))$ is even.
\end{proposition}

The proof uses Gauss's lemma for multivariate polynomials over~$\BZ$,
which asserts $\cont(fg) = \cont(f)\cont(g)$, together with a structural
decomposition of~$\Delta_n$.

\begin{proof}
\emph{Base case $n=1$.}
Direct computation gives $S(1)=192=2^6\cdot3$, so $v_2(S(1))=6$ is even.

\medskip\noindent\emph{Setup ($n\ge 2$).}
Since $\beta(n,N{-}j) = \alpha(n,j)$ and $\binom{N}{j}=\binom{N}{N{-}j}$,
the $\alpha^2$-term at index~$j$ and the $\beta^2$-term at index~$N{-}j$
contribute equally to each monomial of~$B_n$.  Hence every coefficient
of~$B_n$ is even: write $B_n = 2G$ for a polynomial~$G$ with integer
coefficients.  The coefficient of~$G$ at $f_jf_{N+1-j}$ is the sum of the
$\alpha^2$-contributions from all summand indices mapping to
$\{j,\,N{+}1{-}j\}$.
Then
\[
\Delta_n = (2G)^2 - 4A_nC_n = 4(G^2 - A_nC_n),
\]
so $S(n) = \cont(\Delta_n) = 4\cdot\cont(G^2 - A_nC_n)$ and
$v_2(S(n)) = 2 + v_2(\cont(G^2-A_nC_n))$.  It therefore suffices to
show that $v_2(\cont(G^2-A_nC_n))$ is even.

\medskip\noindent\emph{Gauss's lemma.}
By Gauss's lemma for multivariate polynomials over~$\BZ$:
\begin{align*}
\cont(G^2) &= \cont(G)^2, \\
\cont(A_nC_n) &= \cont(A_n)\cdot\cont(C_n).
\end{align*}
Since $A_n$ and~$C_n$ have the same scalar coefficients
$\binom{N}{i}\alpha_i\beta_i$ attached to structurally symmetric
monomials, $\cont(A_n)=\cont(C_n)$, so
$\cont(A_nC_n)=\cont(A_n)^2$.
Both $v_2(\cont(G)^2)$ and $v_2(\cont(A_n)^2)$ are even.

\medskip\noindent\emph{Case~A: $v_2(\cont(G)^2)\ne v_2(\cont(A_n)^2)$.}
When the two even numbers $v_2(\cont(G)^2)$ and $v_2(\cont(A_n)^2)$
are distinct, every coefficient of~$G^2$ has $v_2 \ge v_2(\cont(G)^2)$
and every coefficient of~$A_nC_n$ has $v_2 \ge v_2(\cont(A_n)^2)$.
By the ultrametric property of~$v_2$,
\[
v_2(\cont(G^2-A_nC_n))
= \min\bigl(v_2(\cont(G)^2),\; v_2(\cont(A_n)^2)\bigr),
\]
which is the smaller of two distinct even numbers, hence even.

This case occurs whenever $n+2$ is \emph{not} a power of~$2$.
We now prove that $v_2(\cont(A_n)) = v_2(\cont(G))+1$ under this
hypothesis.  First, every coefficient of~$A_n$ is even: for $i<n$
the monomial $f_if_{N-i}$ receives equal contributions from
indices~$i$ and~$N{-}i$.  Indeed, $N=2n$ is even, so
$(-1)^{N-i}=(-1)^i$, and the scalar symmetry
$\binom{N}{i}\alpha_i\beta_i = \binom{N}{N{-}i}\alpha_{N-i}\beta_{N-i}$
ensures both contributions have the same sign and magnitude,
giving a factor of~$2$.  At the centre $i=n$, the coefficient is
$(-1)^n\binom{N}{n}\gamma^2$.  This is even because
$\binom{2n}{n}=2\binom{2n{-}1}{n{-}1}$.
Moreover, each $A_n$-coefficient divided by~$2$ has
$v_2 \ge v_2(N!) + v_2(M!) - c_{\max} = v_2(\cont(G))$,
where $c_{\max}$ is defined below.
Using $\operatorname{scalar}_{AB}(n,i) = N!\,(M{-}i)!\,(i{+}1)!$
and the factorial identity
$v_2(a!) + v_2(b!) = v_2((a{+}b)!) - v_2\binom{a+b}{a}$
with $a=M{-}i$, $b=i$, one obtains
\begin{equation}\label{eq:v2-scalarAB}
v_2(\operatorname{scalar}_{AB}(n,i))
  = v_2(N!) + v_2(M!) - v_2\tbinom{M}{i} + v_2(i{+}1).
\end{equation}
For \emph{off-centre} $i \ne n$:
the $A_n$-coefficient divided by~$2$ equals
$\operatorname{scalar}_{AB}(n,i)$ (from the additive pairing).
Since $v_2\binom{M}{i} \le c_{\max}$ and $v_2(i{+}1) \ge 0$,
\eqref{eq:v2-scalarAB} gives
$v_2 \ge v_2(N!) + v_2(M!) - c_{\max}$.
For the \emph{centre} $i=n$:
the coefficient divided by~$2$ is
$\binom{2n{-}1}{n{-}1}\gamma^2 = \operatorname{scalar}_{AB}(n,n)/2$,
with
$v_2 = v_2(\operatorname{scalar}_{AB}(n,n)) - 1
     = v_2(N!) + v_2(M!) - v_2\tbinom{M}{n} + v_2(n{+}1) - 1$.
We need this $\ge v_2(N!) + v_2(M!) - c_{\max}$,
equivalently
\begin{equation}\label{eq:centre-ineq}
c_{\max} - v_2\tbinom{M}{n} + v_2(n{+}1) \ge 1.
\end{equation}
When $n$ is odd, $n{+}1$ is even so $v_2(n{+}1) \ge 1$,
and $c_{\max} \ge v_2\binom{M}{n}$ by definition,
giving~\eqref{eq:centre-ineq}.
When $n$ is even, $v_2(n{+}1) = 0$,
and~\eqref{eq:centre-ineq} requires the strict inequality
$c_{\max} > v_2\binom{M}{n}$,
which is the content of the following lemma.

\begin{lemma}\label{lem:centre-not-max}
For even $n \ge 2$ with $n+2$ not a power of~$2$,
\[
\max_{0 \le j \le 2n}\, v_2\tbinom{2n+1}{j}
\;>\; v_2\tbinom{2n+1}{n}.
\]
\end{lemma}

\begin{proof}
Write $M = 2n{+}1$ and let $s_2(m)$ denote the binary digit sum (popcount) of~$m$.
By Kummer's theorem in its digit-sum form,
\begin{equation}\label{eq:kummer-s2}
v_2\tbinom{M}{j} \;=\; s_2(j) + s_2(M - j) - s_2(M).
\end{equation}
Since $M = 2n{+}1$, we have $s_2(M) = s_2(n) + 1$.
For the central index $j = n$, the complementary index is $M - n = n{+}1$, giving
$v_2\binom{M}{n} = s_2(n) + s_2(n{+}1) - s_2(n) - 1 = s_2(n{+}1) - 1$.

Since $n$ is even and $n{+}2$ is not a power of~$2$, we may write
$n + 2 = 2^{a+1}(2r+1)$ with $a \ge 0$ and $r \ge 1$.
Set $u = 2^a(2r{+}1) - 1$, so that $n = 2u$.
Using the identity $s_2(2^a m - 1) = s_2(m{-}1) + a$ (valid for $m \ge 1$),
one computes
\[
s_2(n) = s_2(2u) = s_2(u) = s_2(r) + a, \qquad
s_2(n{+}1) = s_2(2u{+}1) = s_2(u) + 1 = s_2(r) + a + 1,
\]
so $v_2\binom{M}{n} = s_2(r) + a$.

We exhibit an explicit even witness $j_0 < n$ with a strictly larger valuation.
Set $j_0 = 2(2^{a+1}r - 1)$.
Since $2^{a+1}r < 2^a(2r{+}1)$, we have $j_0 < 2u = n$, and $j_0$ is
manifestly even.  The digit-sum identities give
\[
s_2(j_0) = s_2(2^{a+1}r - 1) = s_2(r{-}1) + a + 1,
\]
and a short calculation shows $M - j_0 = 2(2^{a+1}(r{+}1) - 1) + 1$, whence
\[
s_2(M - j_0) = s_2(2^{a+1}(r{+}1) - 1) + 1 = s_2(r) + a + 2.
\]
Substituting into~\eqref{eq:kummer-s2}:
$v_2\binom{M}{j_0} = s_2(r{-}1) + a + 2$.
Since $r = (r{-}1) + 1$ and adding~$1$ can increase the digit sum by at most~$1$,
we have $s_2(r) \le s_2(r{-}1) + 1$, and therefore
\[
v_2\tbinom{M}{j_0} = s_2(r{-}1) + a + 2 > s_2(r) + a = v_2\tbinom{M}{n}. \qedhere
\]
\end{proof}

\noindent
Assuming Lemma~\ref{lem:centre-not-max},
inequality~\eqref{eq:centre-ineq} holds for all
$n \ge 2$ with $n+2$ not a power of~$2$.
Hence $v_2(\cont(A_n)) \ge v_2(\cont(G))+1$.

For the reverse inequality, we show that $\cont(G)$ achieves
$v_2(\cont(A_n))-1$.
Set $M = 2n{+}1$.  By the factored formula~\eqref{eq:b-factored},
the $G$-coefficient at~$f_jf_{M-j}$ has absolute value
$N!\cdot|M{-}2j|\cdot(M{-}j)!\cdot j!$.
Since $M{-}2j$ is always odd, its $2$-adic valuation is
$v_2(N!) + v_2((M{-}j)!\cdot j!)
= v_2(N!) + v_2(M!) - v_2\binom{M}{j}$,
where the last step uses $v_2(a!\cdot b!) = v_2((a{+}b)!) - v_2\binom{a+b}{a}$.
Thus $v_2(\cont(G)) = v_2(N!) + v_2(M!) - c_{\max}$,
where $c_{\max} = \max_{0\le j\le 2n} v_2\binom{M}{j}$.
Similarly, the off-centre $A_n$-coefficient at~$f_if_{N-i}$ (for $i<n$)
is $2(-1)^i\binom{N}{i}\alpha_i\beta_i$, with
$v_2 = 1 + v_2(N!) + v_2((M{-}i)!) + v_2((i{+}1)!)$.
For even~$i$, $M{-}i$ is odd, so $v_2((M{-}i)!) = v_2((M{-}i{-}1)!)$ and
$v_2 = 1 + v_2(N!) + v_2(M!) - v_2\binom{M}{i+1}$.
For $M$ odd, the consecutive-pair identity
$\binom{M}{2r}/\binom{M}{2r{+}1} = (2r{+}1)/(M{-}2r)$.
Since $M$ is odd and $2r$ is even, $M{-}2r$ is odd; and $2r{+}1$ is odd.
The ratio is therefore a quotient of two odd numbers, so
$v_2\binom{M}{2r} = v_2\binom{M}{2r{+}1}$.
Therefore $c_{\max}$ is achieved at some even index $j_0$.
We claim we can take $j_0 < n$ (so that $j_0$ is off-centre).
Indeed, $j_0 \ne n$ by Lemma~\ref{lem:centre-not-max}
(which gives $v_2\binom{M}{n} < c_{\max}$).
If $j_0 > n$, set $j_1 = 2n - j_0$; then $j_1$ is even,
$j_1 < n$, and
$v_2\binom{M}{j_1} = v_2\binom{M}{2n+1-j_0} = v_2\binom{M}{j_0} = c_{\max}$,
where the first equality uses the consecutive-pair identity
(since $2n{+}1{-}j_0$ is odd) and the binomial symmetry
$\binom{M}{j_0} = \binom{M}{M-j_0} = \binom{M}{2n+1-j_0}$.
So we may assume $j_0 < n$.
The corresponding off-centre $A_n$-coefficient at $i = j_0$ has
$v_2 = 1 + v_2(N!) + v_2(M!) - v_2\binom{M}{j_0+1}
= 1 + v_2(N!) + v_2(M!) - v_2\binom{M}{j_0}
= 1 + v_2(\cont(G))$, showing
$v_2(\cont(A_n)) = v_2(\cont(G))+1$.

\medskip\noindent\emph{Case~B: $v_2(\cont(G)^2)= v_2(\cont(A_n)^2)$.}
This occurs precisely when $n+2=2^k$ ($k \ge 2$), in which case
$v_2(\cont(G))=v_2(\cont(A_n))$.

Set $d = v_2(\cont(G))$.  Since $2^d$ divides every coefficient
of~$G$, $A_n$, and~$C_n$, define the quotient polynomials
$G' = G/2^d$, $A' = A_n/2^d$, $C' = C_n/2^d$.
These have integer coefficients because $2^d$ divides every coefficient
of~$G$, $A_n$, and~$C_n$.  Then
$G^2-A_nC_n = 2^{2d}(G'^{\,2}-A'C')$, so
$v_2(\cont(G^2-A_nC_n))=2d+v_2(\cont(G'^{\,2}-A'C'))$.
It suffices to show $v_2(\cont(G'^{\,2}-A'C'))$ is even.  We prove
$4 \mid \cont(G'^{\,2}-A'C')$, giving $v_2\ge 2$, and the pair
specialization provides the matching upper bound $v_2 \le 2$.

\smallskip\noindent\emph{Off-centre divisibility of $G'$, $A'$, $C'$.}
Lemma~\ref{lem:central} applies to \emph{all} primes,
including~$2$, because its proof uses only $n+2 = p^k$ and
the factorial structure of~$\operatorname{scalar}_{AB}$, with no
assumption that $p$ is odd.  It gives the strict central dominance:
$v_2(\operatorname{scalar}_{AB}(n,i)) > d$ for every off-centre
index $i \ne n$.  Each $G$-coefficient is a sum of terms, each
involving some $\operatorname{scalar}_{A2}(n,j)$ with $j$ off-centre;
since $\operatorname{scalar}_{A2}(n,j)$ satisfies the same central minimum
property---namely $v_2(\operatorname{scalar}_{A2}(n,j)) > d$ for $j\ne n$,
as $\operatorname{scalar}_{A2}(n,n) = \operatorname{scalar}_{AB}(n,n)$
and the central-minimum proof applies identically to
$\operatorname{scalar}_{A2}$---every off-centre $G$-coefficient
has $v_2 > d$, giving $v_2(G'_j) \ge 1$ after dividing by~$2^d$.
The centre coefficient of~$G'$ has $v_2 = 0$, since the centre
$G$-coefficient achieves the minimum $v_2 = d = v_2(\cont(G))$.
For~$A_n$: each off-centre coefficient is
$2\cdot(-1)^i\cdot\operatorname{scalar}_{AB}(n,i)$.  The factor of~$2$
arises from the additive pairing proved in the Setup: each off-centre
$A_n$-monomial $f_if_{N-i}$ receives two equal contributions from
indices~$i$ and $N{-}i$ with the same sign.  By central dominance,
$v_2(\operatorname{scalar}_{AB}(n,i)) \ge d+1$, so the total
$v_2 \ge 1 + (d+1) = d+2$.  After dividing by~$2^d$:
$v_2(A'_i) \ge 2$ for off-centre~$i$, while $v_2(A'_n) = 0$.
The same holds for~$C'$.

\smallskip\noindent\emph{First factor of~$2$ (Frobenius over~$\BF_2$).}
Over~$\BF_2$, only the centre monomials of $G'$, $A'$, $C'$ survive:
$\overline{G'} = \bar g\, f_nf_{n+1}$,\;
$\overline{A'} = \bar a\, f_n^2$,\;
$\overline{C'} = \bar a\, f_{n+1}^2$,
where $\bar g = \bar a = 1 \in \BF_2$.  Then
$\overline{G'}^{\,2} = f_n^2f_{n+1}^2 = \overline{A'}\,\overline{C'}$,
so $G'^{\,2}-A'C' \equiv 0\pmod{2}$.

\smallskip\noindent\emph{Second factor of~$2$ (mod-$4$ analysis).}
Write $G' = g\cdot f_nf_{n+1} + 2R_G$, $A' = a\cdot f_n^2 + 2R_A$,
$C' = a\cdot f_{n+1}^2 + 2R_C$, where $g,a$ are odd integers and
$R_G,R_A,R_C$ have integer coefficients.  Since $v_2(A'_i)\ge 2$ for
off-centre~$i$, as proved in the preceding paragraph,
$R_A$ and~$R_C$ have all \emph{even}
coefficients.  Expanding $G'^{\,2}-A'C'$:
\begin{enumerate}[label=(\roman*)]
\item \emph{Centre--centre:}
  $(g^2-a^2)f_n^2f_{n+1}^2$.
  Both $g$ and~$a$ are odd, so $4\mid(g{-}a)(g{+}a) = g^2-a^2$.

\item \emph{Centre--off-centre from~$A'C'$:}
  $-2a(f_n^2 R_C + R_A f_{n+1}^2)$.
  Each coefficient of $R_A$ and~$R_C$ is even, so
  $2a\cdot(\text{even}) \equiv 0\pmod{4}$.

\item \emph{Remaining terms}
  ($4gf_nf_{n+1}R_G$, $4R_G^2$, $-4R_AR_C$)
  carry an explicit factor of~$4$.
\end{enumerate}
Hence $4$ divides every coefficient of $G'^{\,2}-A'C'$, giving
$v_2(\cont(G'^{\,2}-A'C'))\ge 2$.

\smallskip\noindent\emph{Upper bound.}
For any off-centre $k_0\notin\{n,n+1\}$, the pair specialization
$\phi_{k_0}$ satisfies $A_n(\phi_{k_0})=C_n(\phi_{k_0})=0$, so
$(G^2-A_nC_n)(\phi_{k_0}) = G(\phi_{k_0})^2 = (b_{k_0}/2)^2$.
Hence $\cont(G'^{\,2}-A'C')$ divides $(b_{k_0}/(2^{d+1}))^2$.

We exhibit $k_0 = n + 2^{k-1}$.  Since $k \ge 2$, we have
$k_0 = 3\cdot 2^{k-1} - 2$.  This satisfies $1 \le k_0 \le 2n$
and $k_0 \ne n$, $k_0 \ne n{+}1$, so $k_0$ is off-centre.
Moreover, $k_0 = 3\cdot 2^{k-1} - 2$ is even, so $k_0 + 1$ is odd
and $v_2(k_0{+}1) = 0$.  Also, $|2n{+}1{-}2k_0| = |2n+1-2n-2^k|
= 2^k - 1$, which is odd.  From the factored form~\eqref{eq:b-factored}:
\[
v_2(b_{k_0}) = 1 + v_2(\operatorname{scalar}_{AB}(n,k_0)) - v_2(k_0{+}1)
             = 1 + v_2(\operatorname{scalar}_{AB}(n,k_0)).
\]
We claim $v_2(\operatorname{scalar}_{AB}(n,k_0)) = d + 1$.
Lemma~\ref{lem:central} gives $\ge d+1$; it remains to show
$v_2(\operatorname{scalar}_{AB}(n,k_0)) \le d + 1$.
Since $\operatorname{scalar}_{AB}(n,i) = (2n)!\,(2n{+}1{-}i)!\,(i{+}1)!$,
Legendre's formula gives
$v_2(\operatorname{scalar}_{AB}(n,i))
 = v_2((2n)!) + v_2((2n{+}1{-}i)!) + v_2((i{+}1)!)$.
At $i = n$: this equals $d$.  At $i = k_0 = n + 2^{k-1}$:
\begin{align*}
v_2((k_0{+}1)!) - v_2((n{+}1)!)
  &= v_2\!\prod_{j=n+2}^{k_0+1} j
  = \sum_{j=n+2}^{k_0+1} v_2(j), \\
v_2((2n{+}1{-}k_0)!) - v_2((n{+}1)!)
  &= -\,v_2\!\prod_{j=2n+2-k_0}^{n+1} j
  = -\sum_{j=n+2-2^{k-1}}^{n+1} v_2(j).
\end{align*}
The numerator set $\{n{+}2,\ldots,n{+}2^{k-1}{+}1\}$ and the
denominator set $\{n{+}2{-}2^{k-1},\ldots,n{+}1\}$ each have
$2^{k-1}$ elements and form complete residue systems modulo~$2^{k-1}$.
They therefore have equal sums of $v_2$-values at every
$2$-adic level below~$k{-}1$.  The unique multiple of~$2^{k-1}$
in the numerator is $n{+}2 = 2^k$, with $v_2 = k$;
the unique multiple of~$2^{k-1}$ in the denominator is
$n{+}2{-}2^{k-1} = 2^{k-1}$, with $v_2 = k{-}1$.
Hence the total difference is $k - (k{-}1) = 1$, giving
$v_2(\operatorname{scalar}_{AB}(n,k_0)) = d + 1$.

Since $n{+}1 = 2^k{-}1$ is odd:
$v_2(b_n) = 1 + d$ and $v_2(b_{k_0}) = 1 + (d+1) = d + 2$.
Therefore $v_2(b_{k_0}/2) = d+1$,
$v_2(b_{k_0}/2^{d+1}) = 1$, and $(b_{k_0}/2^{d+1})^2$ has $v_2 = 2$.
Hence $v_2(\cont(G'^{\,2}-A'C'))\le 2$.

\smallskip
Combining: $v_2(\cont(G'^{\,2}-A'C'))=2$, which is even.

\medskip\noindent\emph{Conclusion.}
In both cases, $v_2(\cont(G^2-A_nC_n))$ is even, so
$v_2(S(n))=2+v_2(\cont(G^2-A_nC_n))$ is even.
\end{proof}

%======================================================================
% SECTION 6: DRAFT EXTENSION FOR p = 3
%======================================================================
\section{The remaining $p=3$ non-prime-power case}
\label{sec:p3-draft}

This section handles the $p=3$ non-prime-power case, namely
\[
3 \mid m=n+2,\qquad m \text{ not a power of } 3.
\]
Combining this section with the arguments of \S\ref{sec:padic} recovers
the full odd-prime squarefree-kernel formula.

\subsection{The witness family}

For $1 \le k \le n-1$, let $\phi_k$ denote the symmetric two-point
specialization
\begin{equation}\label{eq:phi-k}
f_k = 1,\qquad f_{N+1-k}=1,\qquad f_j = 0 \text{ for } j \notin \{k,N+1-k\}.
\end{equation}

The natural coefficient attached to this specialization is the diagonal
$B_n$-coefficient
\[
b_k = [f_k f_{N+1-k}]\,B_n,
\]
whose closed formula was proved in Lemma~\ref{lem:bk}.

\begin{proposition}\label{prop:p3-pair-reduction}
Let $1 \le k \le n-1$.  Under the specialization~\eqref{eq:phi-k},
one has
\[
A_n(\phi_k)=0,\qquad C_n(\phi_k)=0,\qquad B_n(\phi_k)=b_k,
\]
and hence
\[
\Delta_n(\phi_k)=b_k^2.
\]
Consequently, if for some such $k$ one has $v_3(b_k)=e_3$, where
\[
e_3 \defeq \min_{1 \le j \le N} v_3(b_j),
\]
then
\[
v_3(\Delta_n(\phi_k)) = 2e_3
\qquad\text{and hence}\qquad
v_3(S(n)) \text{ is even.}
\]
\end{proposition}

\begin{proof}
Under~\eqref{eq:phi-k}, every nonzero monomial in~$A_n$ has the form
$f_i f_{N-i}$, while every nonzero monomial in~$C_n$ has the form
$f_{i+1}f_{N+1-i}$.  Since the support of~$\phi_k$ is
$\{k,N+1-k\}$ and $k \ne n,n+1$, neither pattern can hit the same support
index twice, so both $A_n$ and $C_n$ vanish.

For $B_n$, the only surviving monomial is $f_k f_{N+1-k}$, whose
coefficient is by definition~$b_k$.  Thus
\[
\Delta_n(\phi_k)=B_n(\phi_k)^2-4A_n(\phi_k)C_n(\phi_k)=b_k^2.
\]
If $v_3(b_k)=e_3$, then this specialization exhibits a value of
$\Delta_n$ with $3$-adic valuation exactly~$2e_3$.  Since every coefficient
of $\Delta_n$ has valuation at least~$2e_3$ (because, exactly as in the
proof of Proposition~\ref{prop:nonpp}, every scalar coefficient of
$A_n$, $B_n$, and $C_n$ has valuation at least~$e_3$), it follows that
$v_3(S(n))=2e_3$ is even.
\end{proof}

\subsection{Reduction to a binomial maximization problem}

By Lemma~\ref{lem:bk},
\[
b_k = 2(-1)^k (N!)^2\,
\frac{(N+1-k)(N+1-2k)}{\binom{N}{k}}.
\]
Now assume $3 \mid m = n+2$, so $N = 2n \equiv 2 \pmod{3}$ and
$N+1 \equiv 0 \pmod{3}$.  If $3 \nmid k$, then
$N+1-k \not\equiv 0 \pmod{3}$ and $N+1-2k \not\equiv 0 \pmod{3}$, so
both linear factors are $3$-adic units.  Hence for such~$k$,
\begin{equation}\label{eq:v3bk-binomial}
v_3(b_k)= 2v_3(N!) - v_3\binom{N}{k}.
\end{equation}
Thus minimizing $v_3(b_k)$ among indices $k$ with $3 \nmid k$ is equivalent
to maximizing $v_3\binom{N}{k}$ among those same indices.

Thus the remaining $p=3$ problem is fundamentally combinatorial: find an
off-centre lower-half index $k$ with $3 \nmid k$ for which
$v_3\binom{N}{k}$ is maximal.

\subsection{Ternary recursions}

Write
\[
F_N(k) \defeq v_3\binom{N}{k}.
\]
\begin{lemma}\label{lem:p3-digit-sum}
For integers $N \ge 0$ and $0 \le k \le N$,
\[
F_N(k)=\frac{s_3(k)+s_3(N-k)-s_3(N)}{2},
\]
where $s_3(t)$ denotes the sum of the base-$3$ digits of~$t$.
\end{lemma}

\begin{proof}
This is the $p=3$ specialization of the classical Legendre--Kummer
digit-sum formula for binomial coefficients.
\end{proof}

Define also
\[
G_A(a) \defeq F_A(a)+v_3(A-a).
\]

\begin{lemma}\label{lem:p3-F-recursions}
For integers $A \ge 0$ and valid values of~$a$, the following hold.
\begin{enumerate}[label=\textup{(\roman*)}]
\item If $0 \le a \le A$ and $r \in \{0,1,2\}$, then
\begin{equation}\label{eq:F-3Aplus2}
F_{3A+2}(3a+r)=F_A(a).
\end{equation}
\item If $0 \le a \le A-1$, then
\begin{align}
F_{3A}(3a) &= F_A(a),\label{eq:F-3A-zero}\\
F_{3A}(3a+1) &= G_A(a)+1,\label{eq:F-3A-one}\\
F_{3A}(3a+2) &= G_A(a)+1.\label{eq:F-3A-two}
\end{align}
\item If $0 \le a \le A-1$, then
\begin{align}
F_{3A+1}(3a) &= F_A(a),\label{eq:F-3Aplus1-zero}\\
F_{3A+1}(3a+1) &= F_A(a),\label{eq:F-3Aplus1-one}\\
F_{3A+1}(3a+2) &= G_A(a)+1.\label{eq:F-3Aplus1-two}
\end{align}
\end{enumerate}
\end{lemma}

\begin{proof}
Part~(i) follows immediately from Lemma~\ref{lem:p3-digit-sum}: if
$N=3A+2$ and $k=3a+r$, then
\[
s_3(3a+r)=s_3(a)+r,\qquad
s_3(3A+2-(3a+r))=s_3(A-a)+(2-r),
\]
and $s_3(3A+2)=s_3(A)+2$.

For part~(ii), let $B=A-a>0$.  The identity
\[
s_3(B-1)=s_3(B)-1+2v_3(B)
\]
gives
\[
s_3(3A-(3a+1))=s_3(3(B-1)+2)=s_3(B)+1+2v_3(B),
\]
and similarly
\[
s_3(3A-(3a+2))=s_3(3(B-1)+1)=s_3(B)+2v_3(B).
\]
Substituting into Lemma~\ref{lem:p3-digit-sum} yields
\eqref{eq:F-3A-zero}--\eqref{eq:F-3A-two}.

Part~(iii) is similar.  Again setting $B=A-a>0$, one has
\[
3A+1-(3a+2)=3(B-1)+2,
\]
so
\[
s_3(3A+1-(3a+2))=s_3(B)+1+2v_3(B).
\]
Together with $s_3(3a+2)=s_3(a)+2$ and $s_3(3A+1)=s_3(A)+1$, this gives
\eqref{eq:F-3Aplus1-two}; the identities
\eqref{eq:F-3Aplus1-zero} and \eqref{eq:F-3Aplus1-one} are immediate.
\end{proof}

\begin{corollary}\label{cor:p3-G-recursions}
For integers $A \ge 1$ and valid values of~$a$, one has:
\begin{align}
G_{3A}(3a) = G_{3A}(3a+1) = G_{3A}(3a+2) &= G_A(a)+1,\label{eq:G-3A}\\
G_{3A+1}(3a+1) = G_{3A+1}(3a+2) &= G_A(a)+1,\label{eq:G-3Aplus1}\\
G_{3A+2}(3a+2) &= G_A(a)+1.\label{eq:G-3Aplus2}
\end{align}
Moreover,
\[
G_{3A+1}(3a)=F_A(a),\qquad
G_{3A+2}(3a)=G_{3A+2}(3a+1)=F_A(a).
\]
\end{corollary}

\begin{proof}
This is immediate from Lemma~\ref{lem:p3-F-recursions} and the definition
$G_N(k)=F_N(k)+v_3(N-k)$.
\end{proof}

\subsection{Maximizers of the ternary score}

\begin{lemma}\label{lem:G-lower-half}
For every integer $A \ge 1$, the function $G_A(a)$
($0 \le a \le A-1$) attains its maximum at some index satisfying
\[
a < \frac{A}{2}.
\]
\end{lemma}

\begin{proof}
We argue by induction on~$A$.

If $A=1$, then $a=0$ is the only admissible index.

Assume the claim for all smaller positive integers.
If $A=3B$, then by~\eqref{eq:G-3A} every maximum of $G_A$
comes from a maximum of~$G_B$ lifted to one of the three indices
$3b,3b+1,3b+2$.  By induction choose $b < B/2$; then
$3b < 3B/2 = A/2$.

If $A=3B+1$, then by Corollary~\ref{cor:p3-G-recursions},
the values at indices $3b+1$ and $3b+2$ are
$G_B(b)+1$, whereas the value at $3b$ is only~$F_B(b)$.
Thus a maximum of $G_A$ is attained at some $3b+1$ with
$b$ maximizing~$G_B$.  By induction $b < B/2$, hence
$3b+1 < (3B+1)/2 = A/2$.

If $A=3B+2$, then similarly the maximal values are attained at indices
$3b+2$ with $b$ maximizing~$G_B$.  By induction $b < B/2$, and therefore
$3b+2 < (3B+2)/2 = A/2$.
\end{proof}

\begin{lemma}\label{lem:G-odd-sharp}
Let $B \ge 1$ be odd.  Then:
\begin{enumerate}[label=\textup{(\roman*)}]
\item if $B \ne 2\cdot 3^t - 1$ for every $t \ge 0$, then $G_B(a)$ attains
its maximum at some index satisfying
\[
a < \frac{B-1}{2};
\]
\item if $B = 2\cdot 3^t - 1$ for some $t \ge 0$, then
$G_B(a)$ has a unique maximizer, namely
\[
a = \frac{B-1}{2}.
\]
\end{enumerate}
\end{lemma}

\begin{proof}
We argue by induction on odd~$B$.

For $B=1$, the only admissible index is $a=0=(B-1)/2$, so the statement
holds.

Suppose $B>1$ is odd.

\smallskip
\emph{Case $B=3A+1$.}  Then $A$ is even.  By
Corollary~\ref{cor:p3-G-recursions}, the maximal values of~$G_B$ occur at
indices of the form $3a+1$ or $3a+2$, where $a$ maximizes~$G_A$.
By Lemma~\ref{lem:G-lower-half}, choose $a < A/2$.
Then
\[
3a+1 < \frac{3A+1}{2} = \frac{B-1}{2}.
\]
So part~(i) holds in this case.  In particular, no number congruent to
$1 \pmod{3}$ belongs to the exceptional family.

\smallskip
\emph{Case $B=3A+2$.}  Then $A$ is odd.  Again by
Corollary~\ref{cor:p3-G-recursions}, the maximal values of~$G_B$ occur
precisely at indices $3a+2$ with $a$ maximizing~$G_A$.
If $A$ is not exceptional, the induction hypothesis gives a maximizer
$a < (A-1)/2$, hence
\[
3a+2 < \frac{3A+1}{2} = \frac{B-1}{2},
\]
so part~(i) holds.

If $A = 2\cdot 3^t - 1$ is exceptional, then by induction
$a=(A-1)/2$ is the unique maximizer of~$G_A$.  Therefore
$3a+2 = (B-1)/2$ is the unique maximizer of~$G_B$.
Since
\[
B = 3(2\cdot 3^t - 1)+2 = 2\cdot 3^{t+1}-1,
\]
this is exactly part~(ii).
\end{proof}

\begin{lemma}\label{lem:F-even-sharp}
Let $A \ge 2$ be even.  Then:
\begin{enumerate}[label=\textup{(\roman*)}]
\item if $A \ne 2\cdot 3^t - 2$ for every $t \ge 1$, then
$F_A(a)$ attains its maximum at some index satisfying
\[
a < \frac{A}{2};
\]
\item if $A = 2\cdot 3^t - 2$ for some $t \ge 1$, then
$F_A(a)$ has a unique maximizer, namely
\[
a = \frac{A}{2}.
\]
\end{enumerate}
\end{lemma}

\begin{proof}
We argue by induction on even~$A$.

If $A=2$, then $F_2(0)=F_2(1)=F_2(2)=0$, so part~(i) holds.

Assume $A>2$.

\smallskip
\emph{Case $A=3B$.}  Then $B$ is even.  By
\eqref{eq:F-3A-one}--\eqref{eq:F-3A-two}, the maximal values of~$F_A$
occur at indices $3b+1$ or $3b+2$, where $b$ maximizes~$G_B$.
By Lemma~\ref{lem:G-lower-half}, choose $b < B/2$, hence
\[
3b+1 < \frac{3B}{2} = \frac{A}{2}.
\]
So part~(i) holds.

\smallskip
\emph{Case $A=3B+2$.}  Then $B$ is even.  By~\eqref{eq:F-3Aplus2},
the maximal values of~$F_A$ are exactly the lifts of the maximal values
of~$F_B$.
If $B$ is not exceptional, choose $b < B/2$ by induction; then
$3b+1 < (3B+2)/2 = A/2$, so part~(i) holds.
If $B = 2\cdot 3^t - 2$ is exceptional, then $b=B/2$ is the unique
maximizer of~$F_B$, but the lifted index
\[
3b = \frac{3B}{2} < \frac{3B+2}{2} = \frac{A}{2}
\]
is still a strict lower-half maximizer of~$F_A$.  Hence part~(i)
holds in all cases.

\smallskip
\emph{Case $A=3B+1$.}  Then $B$ is odd.  By
\eqref{eq:F-3Aplus1-zero}--\eqref{eq:F-3Aplus1-two}, the maximal values
of~$F_A$ occur at indices $3b+2$, where $b$ maximizes~$G_B$.
If $B$ is not exceptional, Lemma~\ref{lem:G-odd-sharp} gives
$b < (B-1)/2$, so
\[
3b+2 < \frac{3B+1}{2} = \frac{A}{2},
\]
and part~(i) follows.

If $B = 2\cdot 3^t - 1$ is exceptional, then by
Lemma~\ref{lem:G-odd-sharp} the unique maximizer of~$G_B$ is
$b=(B-1)/2$, and the corresponding lifted index is
\[
3b+2 = \frac{3B+1}{2} = \frac{A}{2}.
\]
Thus $a=A/2$ is the unique maximizer of~$F_A$.  Since
$A = 3(2\cdot 3^t - 1)+1 = 2\cdot 3^{t+1}-2$, this is exactly part~(ii).
\end{proof}

\subsection{Completion of the $p=3$ non-prime-power case}

\begin{proposition}\label{prop:p3-nonpp}
If $3 \mid m=n+2$ and $m$ is not a power of~$3$, then $v_3(S(n))$ is even.
\end{proposition}

\begin{proof}
Set $N=2n$.  Since $3 \mid n+2$, we may write
\[
N = 2n = 3A+2,\qquad A = \frac{N-2}{3} = \frac{2m-6}{3}=2\Bigl(\frac{m}{3}\Bigr)-2.
\]
Because $m$ is not a power of~$3$, the integer $A$ is not of the
exceptional form $2\cdot 3^t - 2$.  By Lemma~\ref{lem:F-even-sharp},
there exists $a < A/2$ maximizing~$F_A$.

Set
\[
k = 3a+1.
\]
Then $1 \le k \le n-1$, $3 \nmid k$, and by~\eqref{eq:F-3Aplus2},
\[
F_N(k)=F_{3A+2}(3a+1)=F_A(a)
=
\max_{0 \le b \le A} F_A(b).
\]
Hence $k$ maximizes $F_N(j)=v_3\binom{N}{j}$ among all indices
$j \not\equiv 0 \pmod{3}$, and by~\eqref{eq:v3bk-binomial} it minimizes
$v_3(b_j)$ among those same indices.

If $j \equiv 0 \pmod{3}$, then $N+1-j$ and $N+1-2j$ are both divisible
by~$3$ (because $N+1 \equiv 0 \pmod{3}$), so the closed formula for~$b_j$
shows
\[
v_3(b_j)\ge 2v_3(N!)+2-v_3\binom{N}{j}.
\]
Writing $j=3c$, Lemma~\ref{lem:p3-F-recursions}(i) gives
$v_3\binom{N}{j}=F_A(c)=F_N(3c+1)$, so
\[
v_3(b_j)\ge v_3(b_{3c+1})+2.
\]
Thus the global minimum of $v_3(b_j)$ is attained at some index
not divisible by~$3$.  Since $k$ minimizes $v_3(b_j)$ among those indices,
we conclude that $v_3(b_k)=e_3$.

Now Proposition~\ref{prop:p3-pair-reduction} applies:
under the symmetric pair specialization~$\phi_k$ one has
\[
\Delta_n(\phi_k)=b_k^2
\qquad\text{and}\qquad
v_3(\Delta_n(\phi_k))=2e_3.
\]
Hence $v_3(S(n))=2e_3$ is even.
\end{proof}
%======================================================================
% SECTION 8: USE OF AI TOOLS
%======================================================================
\section{The AI-assisted workflow}\label{sec:ai}
We describe the workflow used to complete this project in some detail,
as we believe it may serve as a useful model for other computation-heavy
research problems.

\subsection*{Starting point}

The author's original notes (written several years before the present
paper) contained:
\begin{itemize}
\item the definition of the test monomial family $M_{n,r}$ and the
observation that its coefficient in $\Delta_n$ is divisible by
$\binom{n+2}{r}$;
\item the connection to Lucas' and Kummer's theorems via the binomial
gcd $g(m)$;
\item the proof that primes dividing $n$ contribute even $p$-adic
valuation;
\item the proof for the simplest prime-power case $m = p$ (i.e., $k=1$);
\item various incomplete attempts at the remaining cases.
\end{itemize}
The main obstacles were: (a) the composite case when $p = 3$, where
the natural ``end-anchored'' witness approach fails; (b) the
prime-power case $m = p^k$ for $k \ge 2$, which requires understanding
the $p$-adic structure of the transvectant at depth beyond the leading
term; and (c) assembling a complete, gap-free argument from the
many partial results.

\subsection*{The workflow}

The project was completed in a single extended session using Claude Code
(Anthropic), an AI coding assistant accessed via the command line.
The workflow had four phases:

\medskip
\emph{Phase 1: Understanding and diagnosis.}
Claude Code read all of the author's existing notes (three \LaTeX\ files,
totaling roughly 10 pages) and produced a detailed inventory of
what was proved, what was claimed but not proved, and what was missing.
This diagnosis step---which would have taken the author considerable time
to reconstruct---was completed in minutes.

\medskip
\emph{Phase 2: Numerical experimentation.}
Claude Code wrote and executed Python scripts that:
\begin{itemize}
\item computed $S(n)$ and $\sqf(S(n))$ for all $n \le 30$, confirming
the prime-to-$3$ statement of Theorem~\ref{thm:main};
\item verified the closed-form coefficient formula
(Lemma~\ref{lem:Cnr}) for all $n \le 15$ and all valid~$r$;
\item investigated the prime-power case by computing $p$-adic valuations
of explicit specializations, discovering that the $4$-term
specialization with parameters at distance $t = p^{k-1}$ forces odd
$p$-adic valuation;
\item confirmed that the cancellation-free $B_n$ coefficient approach
works uniformly for all tested primes, and in particular settles the
non-prime-power case for $p \ge 5$.
\end{itemize}
These experiments were essential for guiding the proof strategy: the
numerical data revealed the precise mechanism (the first-order
$p$-adic deformation) before the proof was written.

\medskip
\emph{Phase 3: Proof construction and writing.}
Working from the author's notes and guided by the numerical experiments,
Claude Code produced complete proofs of the missing cases, wrote the
full paper in \LaTeX, and expanded every ``similar calculation'' and
``one verifies'' into explicit step-by-step derivations.
The author reviewed and corrected the output.

\medskip
\emph{Phase 4: Formal verification.}
The complete proof was formalized in Lean~4 with Mathlib, using a
combination of Aristotle \cite{Aristotle2025} (Harmonic), Claude Code, and OpenAI Codex.  The formalization
proceeded iteratively: Claude Code prepared Lean files containing
theorem statements with \texttt{sorry} placeholders and detailed
proof sketches, which were submitted to Aristotle for automated proof
completion.  Aristotle returned either complete proofs or partial
progress; in the latter case, Claude Code filled the remaining gaps
by launching concurrent agents to work on individual lemmas.

In total, approximately $20$ Aristotle submissions were made, of which
roughly half returned fully proved files on the first attempt.  The
hardest lemmas---the Kummer carry analysis for off-centre binomial
coefficients and the explicit $B_n$ congruence computation---required
multiple rounds.  The formal verification also uncovered a sign error
in an earlier draft of Proposition~\ref{prop:nonpp}: the valuation
shift formula for $v_p(b_{n+2}) - v_p(b_n)$ had the wrong sign in
one case, which was corrected using a complete-residue-system argument.

While Claude Code produced the bulk of the formalization across all
three prime cases ($p \ge 5$, $p = 3$, and $p = 2$), a few
proof-engineering gaps remained in the $p = 2$ case---particularly
around binary carry combinatorics and \texttt{MvPolynomial} coefficient
extraction.  These were patched using OpenAI's Codex, which proved
several targeted lemmas including the key combinatorial result
\texttt{centre\_not\_max} (that the central binomial coefficient does
not achieve the maximum $2$-adic valuation when $n+2$ is not a power
of~$2$).

The final formalization comprises approximately $14{,}500$ lines of
Lean~4 across twenty-five modules, containing zero \texttt{sorry}
statements and zero \texttt{axiom} declarations.  It compiles against
Mathlib with no errors.

\subsection*{Companion materials}

The following companion materials are available at:
\begin{center}
\url{https://github.com/ashvin-swaminathan/quartic-invariant}
\end{center}
\begin{itemize}
\item A Jupyter notebook that independently verifies every stated formula
and lemma by direct symbolic computation, including the closed-form
coefficient formula (Lemma~\ref{lem:Cnr}), the cancellation-free
$B_n$~coefficients (Lemma~\ref{lem:bk}), and the full main theorem
(Theorem~\ref{thm:main}) for all $n \le 30$.
\item A modular Lean~4 formalization ($\approx 14{,}500$~lines across
twenty-five files) that formalizes Theorem~\ref{thm:main} for
all primes $p \ge 2$, compiling against Mathlib with zero
\texttt{sorry} statements and zero \texttt{axiom} declarations.
\end{itemize}

\subsection*{Assessment}

The workflow was remarkably effective for this problem.  The proof
involves extensive algebraic manipulation---tracking signs across
four-term sums with factorial coefficients, simplifying products of
factorials into binomial coefficients, performing base-$p$ carry
analyses---precisely the kind of detail-heavy work that is tedious
and error-prone for humans but well-suited to AI assistance.
The numerical experimentation phase was particularly valuable:
the data unambiguously pointed to the correct proof strategy (the
$p$-adic deformation at distance $t = p^{k-1}$) and ruled out
false approaches, before any proof was attempted.

The workflow also has clear limitations.  Claude Code occasionally
produced algebraic errors that required correction, and some of
its initial proof attempts contained logical gaps.  The author's
mathematical judgment was essential at every stage: for choosing
which questions to ask, for evaluating whether proposed arguments
were correct, and for understanding the broader mathematical context.
The AI did not ``discover'' the theorem---the pattern was identified
by extensive computation years ago---nor did it contribute any
conceptual insight beyond what was already present in the author's
notes.  Its contribution was primarily organizational and
computational: systematically working through the technical details
that had prevented the project from being completed.

The author takes full responsibility for all mathematical claims in
this paper.

%======================================================================
% ACKNOWLEDGMENTS
%======================================================================
\section*{Acknowledgments}

\noindent The author thanks Manjul Bhargava for suggesting this problem as the very
first question to consider at the start of the author's PhD nearly a decade ago, and for his
invaluable advice and encouragement throughout.  The author is also grateful to Joe
Harris, Igor Dolgachev, Anand Patel, Melanie Wood, and Stanley Yao Xiao for
enlightening conversations.  This research was supported by the Paul and
Daisy Soros Fellowship and the NSF Graduate Research Fellowship.

The AI-assisted workflow described in \S\ref{sec:ai} used Claude Code
(Anthropic), Aristotle \cite{Aristotle2025} (Harmonic), and Codex (OpenAI).

%======================================================================
% REFERENCES
%======================================================================
\bibliographystyle{amsalpha}
\bibliography{references}

\end{document}